\documentclass[preprint,12pt]{elsarticle}

\usepackage{amsthm}
\usepackage{amsmath,amssymb,txfonts}
\usepackage{delarray,enumerate, comment}
\usepackage{color,bm,comment}
\usepackage[english]{babel}

\makeatletter
 \@addtoreset{equation}{section}
\makeatother

\numberwithin{equation}{section}
\newtheorem{theorem}{Theorem}[section]

\newtheorem{lemma}[theorem]{Lemma}
\newtheorem{prop}[theorem]{Proposition}
\newtheorem{definition}[theorem]{Definition}

\newtheorem{remark}[theorem]{Remark}

\numberwithin{theorem}{section}

\def\re{\mathbb{R}}

\def\tcr{\textcolor{red}}
\def\tcb{\textcolor{blue}}
\def\({\left(}
\def\){\right)}
\def\pd{\partial}
\def\ol{\overline}

\def\lap{\Delta}

\def\ep{\varepsilon}
\def\w{\omega}
\def\la{\lambda}

\def\tx{\tilde{x}}
\def\ty{\tilde{y}}

\def\tcr{\relax}
\def\tcb{\relax}
\def\tcrr{\textcolor{red}}
\def\tcbb{\textcolor{blue}}
\def\tcrr{\relax}
\def\tcbb{\relax}

\begin{document}

\begin{frontmatter}



\title{\tcr{Critical} Hardy inequality on the half-space \\via \tcr{the} harmonic transplantation}


\author[MS]{Megumi Sano\corref{Sano}\fnref{label1}}
\ead{smegumi@hiroshima-u.ac.jp}
\fntext[label1]{Corresponding author.}

\author[FT]{Futoshi Takahashi}
\ead{futoshi@sci.osaka-cu.ac.jp}

\address[MS]{Laboratory of Mathematics, School of Engineering,
Hiroshima University, Higashi-Hiroshima, 739-8527, Japan \tcb{/
Mathematical Institute, Tohoku University, Sendai, 980-8578, Japan}}
\address[FT]{Department of Mathematics, Graduate School of Science, Osaka City University, Sumiyoshi-ku, Osaka, 558-8585, Japan}


\date{\today}

\begin{abstract}
We \tcr{prove a} critical Hardy inequality on the half-space $\re^N_+$ by using 
the harmonic transplantation \tcr{for functions in} $\dot{W}_0^{1,N}(\re^N_+)$. 
\tcr{Also} we give an improvement of the subcritical Hardy inequality on $\dot{W}_0^{1,p}(\re^N_+)$ for $p \in [2, N)$\tcr{,} which converges to the critical Hardy inequality 
\tcr{when} $p \nearrow N$. 
Sobolev type inequalities are also discussed. 
\end{abstract}

\end{frontmatter}

\tableofcontents



%
%
\section{Introduction and main results}\label{Intro}

Let $\Omega \subset \re^N$\tcr{, $N \ge 2$,} be a domain with $a \in \Omega$ and $1 < p <N$. 
The  Hardy inequality
\begin{equation}
\label{H_p}
\( \frac{N-p}{p} \)^p \int_{\Omega} \frac{|u|^p}{|x -a|^p} dx \le \int_{\Omega} | \nabla u |^p dx,
\end{equation}
holds for all $u \in \dot{W}^{1,p}_0(\Omega)$, where $\dot{W}_0^{1,p}(\Omega)$ is \tcr{the homogeneous Sobolev space defined as} the completion of $C_c^{\infty}(\Omega)$ 
with respect to the \tcr{(semi-)}norm $\| \nabla (\cdot )\|_{L^p(\Omega)}$. 
It is well-known that $(\frac{N-p}{p})^p$ is the best constant and is not attained. 
Hardy's best constant $\( \frac{N-p}{p} \)^p$ plays an important role \tcr{in investigating qualitative} properties of solutions to elliptic \tcr{or} parabolic partial differential equations, 
\tcr{such as, stability, instantaneous blow-up, and global-in-time asymptotics,} see \tcr{for example} \cite{BG, BV}.

On the other hand, in the limiting case where $p=N$, the Hardy inequality (\ref{H_p}) looks degenerate as the best constant vanishes. 
However \tcr{in this case,} we can obtain the critical Hardy inequality on bounded domains:
\begin{align}\label{H_N}
	\( \frac{N-1}{N} \)^N \int_{\Omega} \frac{|u|^N}{|x-a|^N \( \log \frac{R}{|x-a|} \)^N} dx 
	\le \int_{\Omega} | \nabla u |^p dx \quad \( u \in \dot{W}_0^{1,N}(\Omega)\)
\end{align}
as a limit of the Hardy inequality (\ref{H_p}) as $p \nearrow N$, where $R=\sup_{x \in \Omega} |x-a|$, 
see \cite{I} or I.-(ii) in \S \ref{S Harmonic}.
It is also known that $(\frac{N-1}{N})^N$ is the best constant and is not attained, see e.g. \cite{AS,II, B-TF}.


In the present paper, we introduce \tcr{a} critical Hardy inequality \tcr{similar to} (\ref{H_N}) when $\Omega$ is unbounded, especially the half-space
$\re^N_+ =\{ (x,y) \,| \, x \in \re^{N-1}, y>0 \}$. 
Note that if $\Omega = \re^N$, this kind of inequality does not hold even if we restrict functions to radially symmetric \tcr{ones}, see Proposition \ref{Prop re^N} in \S \ref{S App}. 

Our first result is the following.


\begin{theorem}\label{Thm CH}(\tcr{Critical} Hardy inequality on the half-space)
Let $N \ge 2$. Then the inequality 
\begin{align}\label{H_N half}
 \( \frac{N-1}{N} \)^N \int_{\re^N_+} &\frac{|u(x,y)|^N}{\( |x|^2 +(1-y)^2 \)^{\frac{N}{2}} \( \frac{|x|^2 +(1+y)^2}{4} \)^{\frac{N}{2}}  \( \log \sqrt{ \frac{|x|^2 +(1+y)^2}{|x|^2 +(1-y)^2}} \)^N }  \,dxdy \notag \\
 &\le \int_{\re^N_+} |\nabla u(x,y)|^N \,dxdy
\end{align}
holds for any $u \in \dot{W}_0^{1,N}(\re^N_+)$. Furthermore, $( \frac{N-1}{N})^N$ is the best constant and is not attained. 
\end{theorem}

\begin{remark}(Asymptotic behavior of the potential function)
Set 
\begin{equation}
\label{V_N}
V_N (x, y) := \frac{1}{\( |x|^2 +(1-y)^2 \) \( \frac{|x|^2 +(1+y)^2}{4} \)  \( \log \sqrt{ \frac{|x|^2 +(1+y)^2}{|x|^2 +(1-y)^2}} \)^2 }.
\end{equation}
\tcr{Then the inequality (\ref{H_N half}) is of the form
\begin{align*}
 \( \frac{N-1}{N} \)^N \int_{\re^N_+} V_N(x,y)^{\frac{N}{2}} |u(x,y)|^N \, dxdy \le \int_{\re^N_+} |\nabla u(x,y)|^N \,dxdy.
\end{align*}
The inequality (\ref{H_N half}) has two aspects: 
one is the critical Hardy inequality on bounded domains and the other is the geometric Hardy inequality on $\re^N_+$, 
which involves \tcb{the distance from the boundary} $\pd \re^N_{+}$.
\tcb{Indeed,} the potential function $V_N(x,y)^{\frac{N}{2}}$ behaves like $\( |x|^2 +(1-y)^2 \)^{-\frac{N}{2}} \( \log \frac{2}{\sqrt{|x|^2 +(1-y)^2}}  \)^{-N}$ when $(x, y)$ is near to $e_N = (0,1) \in \re^N_{+}$, 
which is similar to the critical Hardy potential on the ball $B_2(e_N)$ with radius $2$ and center $e_N$. 
Also $V_N(x,y)^{\frac{N}{2}}$ behaves like $y^{-N} = {\rm dist }((x,y), \pd\re^N_{+})^{-N}$ near the boundary $\pd \re^N_+$ or $\infty$, which is similar to the geometric Hardy potential on $\re^N_+$. 
In fact, since $Y:=\frac{|x|^2 +(1+y)^2}{|x|^2 +(1-y)^2}=1 + o(1)$ as $|x|^2 +(y-1)^2 \to \infty$ or $y \to 0$ and $\log Y = Y-1 + o(1)$ as $Y \to 1$, we have
\begin{align*}
V_N(x,y)^{\frac{N}{2}} 
&= \frac{1}{\( |x|^2 +(1-y)^2 \)^{\frac{N}{2}} \( \frac{|x|^2 +(1+y)^2}{4} \)^{\frac{N}{2}}  \( \log \sqrt{Y} \)^N } \\ 
&= \frac{1}{\( |x|^2 +(1-y)^2 \)^{\frac{N}{2}} \( \frac{|x|^2 +(1+y)^2}{4} \)^{\frac{N}{2}}  \( \frac{Y-1}{2} \)^N } + o(1) \\ 
&= \frac{\( |x|^2 +(1-y)^2 \)^{\frac{N}{2}}}{ y^N \( |x|^2 +(1+y)^2 \)^{\frac{N}{2}} } + o(1) \\
& = O(y^{-N})
\end{align*}
}
as $|x|^2 +(y-1)^2 \to \infty$ or $y \to 0$. 
\end{remark}


Next, we give an improvement of (\ref{H_p}) which \tcr{yields} (\ref{H_N half}) as $p \nearrow N$. 
Improvements of the Hardy and the Hardy-Sobolev inequalities on balls are studied for radially symmetric functions in \cite{FKR,I,S(ArXiv)}. 
However, on the half-space $\re^N_+$, we cannot consider \tcr{radial} symmetry 
since radial functions which are zero on the boundary $\pd \re^N_+$ \tcr{must be identically zero.} 
Instead of \tcr{radial} symmetry, we introduce the following \tcr{new} symmetry for functions $u=u(x,y)$ on $\re^N_+$:

\tcr{Put}
\begin{align}\label{def U}
U_p (x,y)=  
	\begin{cases}
	\frac{p-1}{N-p} \w_{N-1}^{-\frac{1}{p-1}} 
	\left[ \( |x|^2 +(1-y)^2 \)^{-\frac{N-p}{2(p-1)}} - \( |x|^2 +(1+y)^2 \)^{-\frac{N-p}{2(p-1)}} \right] 
	\, &\text{if} \, p \in (1,N), \\
	\w_{N-1}^{-\frac{1}{N-1}} \log \sqrt{ \frac{ |x|^2 +(1+y)^2}{ |x|^2 +(1-y)^2}} 
	&\text{if}\, p=N,
	\end{cases}
\end{align}
\tcr{where} $\omega_{N-1}$ is the area of the unit sphere $\mathbb{S}^{N-1}$ in $\re^N$.
\tcr{Note that the function $U_p(x,y)$ is obtained from the fundamental solution of the $p$-Laplacian on $\re^N$ with singularity $e_N = (0, 1)$, by ``reflecting the singularity" with respect to the boundary $\pd \re^N_{+}$.
So $U_p \equiv 0$ on $\pd \re^N_{+}$.
However, it is different from $p$-Green's function $G_{\re^N_+, e_N}(x,y)$ on $\re^N_{+}$ when $p \not= 2$ and $p \not= N$, see \S \ref{S Green}.}

\tcr{We consider functions on $\re^N_{+}$ of the form}
\begin{align}\label{*'}
	u(x,y) = \tilde{u}(s),\,\text{where}\,\, s= U_p(x,y), \quad \tcr{(x, y) \in \re^N_{+},}
\end{align}
\tcr{for some function $\tilde{u}$ on $\re$ with the property $\tilde{u}(0) = 0$.}
\tcr{In the following, with some ambiguity, we identify $\tilde{u}$ as $u$ and write, for example, $u(x, y) = u(s)$, $s = U_p(x,y)$ for $(x, y) \in \re^N_{+}$.}
\tcr{Thus a function of the form \eqref{*'} has the same value on each level set of $U_p$ and vanishes on $\pd \re^N_{+}$.}

Our second result is an improvement of the Hardy inequality (\ref{H_p}) on $\re^N_+$ for functions with the symmetry (\ref{*'}).  


\begin{theorem}(\tcr{Improved} Hardy inequality for $p \ge 2$)
\label{Thm IH}
Let $2 \le p < N$. 
Then the inequality
\begin{align}
\label{IH}
	\( \frac{N-p}{p} \)^p &\int_{\re^N_+} \frac{V_p(x,y)^{\frac{p}{2}}}{\( |x|^2 +(1-y)^2 \)^{\frac{p}{2}}} |u(x,y)|^p \,dxdy 
	\le \int_{\re^N_+} |\nabla u(x,y)|^p \,dxdy
\end{align}
holds for any $u \in \dot{W}_0^{1,p} (\re^N_+)$ \tcr{of the form} (\ref{*'}), 
where 
\tcr{
\begin{align}
\label{V_p}
\begin{cases}
	&V_p(x,y) = \frac{1+ X^{\frac{N-1}{p-1}} -2X^{\frac{N-p}{2(p-1)}} \( |x|^2 +(1+y)^2 \)^{\tcrr{-1}} (|x|^2 +y^2 -1)}
	{\left[ 1-X^{\frac{N-p}{2(p-1)}} \right]^{2}}, \\
	&X = \frac{|x|^2 +(1-y)^2}{|x|^2 + (1+y)^2} \in [0, 1).
\end{cases}
\end{align}
}
Furthermore, $(\frac{N-p}{p})^p$ is the best constant and is not attained.
\end{theorem}

\begin{remark}\label{Rem without sym}
\tcrr{Actually, the inequality (\ref{IH}) holds for functions without the symmetry (\ref{*'}) by combining Proposition \ref{Prop sol U} and a result in \cite{DD}, see Theorem \ref{Thm IH without sym} in \S \ref{S Proof}. Our method is based on the harmonic transplantation.} 
\end{remark}

\begin{remark}($V_p \ge 1$)
\tcr{We remark here} that (\ref{IH}) is an improvement of the Hardy inequality (\ref{H_p}) \tcr{with $a = e_N$} since $V_p (x, y) \ge 1$. 
In fact, for any $(x,y) \in \re^N_+ \cap \overline{B_1^N}$, we have 
$V_p(x,y) \ge 1+ X^{\frac{N-1}{p-1}} \ge1$. 
Also, for any $(x,y) \in \re^N_+ \setminus \overline{B_1^N}$, 
we have 
\begin{align*}
	V_p(x,y) &=  
	1+ \frac{X^{\frac{N-p}{p-1}}}{(1-X^{\frac{N-p}{2(p-1)}}  )^2} 
\left[ X-1+ X^{- \frac{N-p}{2(p-1)}}\frac{4(y+1)}{|x|^2 + (1+y)^2} \right]   \\
	&= 1+ \frac{X^{\frac{N-p}{p-1}}}{(1-X^{\frac{N-p}{2(p-1)}}  )^2} 
\left[ \left\{ X^{- \frac{N-p}{2(p-1)}} -1 \right\} \frac{4y}{|x|^2 + (1+y)^2} + X^{- \frac{N-p}{2(p-1)}}\frac{4}{|x|^2 + (1+y)^2} \right] \\
	&\ge 1 
\end{align*} 
\tcr{since $X \in [0,1)$.}
\end{remark}

\begin{remark}
Unlike (\ref{H_p}), it is possible to take the limit \tcr{$p \nearrow N$ in} the improved Hardy inequality (\ref{IH}).
\tcr{Thus we obtain Theorem \ref{Thm CH} (for functions with symmetry \eqref{*'}) from Theorem \ref{Thm IH} in this way.}

In fact, since $1-X^s = s \log \frac{1}{X} + \tcr{o(s)}$ as $s \to 0$, 
\tcr{taking $s = \frac{N-p}{2(p-1)}$,} we see that \tcr{$V_p$ in \eqref{V_p} satisfies}
\begin{align*}
	&\( \frac{N-p}{p} \)^p \frac{V_p(x,y)^{\frac{p}{2}}}{\( |x|^2 +(1-y)^2 \)^{\frac{p}{2}} }  \\
	&= \( \frac{N-p}{p} \)^p 
	\frac{\left\{ 1+X -2 \( |x|^2 +(y+1)^2 \)^{-1} (|x|^2+ y^2  -1) \right\}^{\frac{N}{2}}}{ \( |x|^2 +(1-y)^2 \)^{\frac{N}{2}}  
	\left[ \frac{N-p}{2(p-1)} \log \frac{1}{X} \right]^{p}} + o(1) \\
	&= \( \frac{N-1}{N} \)^N V_N (x,y)^{\frac{N}{2}} + o(1) \quad (p \nearrow N)
\end{align*}
\tcr{where $V_N$ is defined in \eqref{V_N}.}
Therefore, we obtain (\ref{H_N half}) as a limit of (\ref{IH}) as $p \nearrow N$.

\tcr{However, note that, Theorem \ref{Thm CH} is proved by another method in \S \ref{S Proof} and valid for functions without any symmetry.} 
\end{remark}


This paper is organized as follows:
In \S \ref{S Green}, we show propositions about the function $U_p(x,y)$ in \eqref{def U}, which coincides with \tcr{$p$-}Green's function \tcr{$G_{\re^N_+, e_N}(x,y)$} when $p=2$ or $N$. 
Although $U_p$ is \tcr{different from $G_{\re^N_+, e_N}$} for $p \in (2, N)$, \tcr{we can prove that} $U_p$ is superharmonic for $p \in (2, N)$ on \tcr{$\re^N_+ \setminus \{ e_N \}$.} 
This is a key point of the proof of Theorem \ref{Thm IH}.
In \S \ref{S Transformation}, we recall \tcr{the} M\"obius transformation and \tcr{the} harmonic transplantation proposed by Hersch \cite{H}. 
We point out that various transformations so far \tcr{appeared in references} can be understood as a special or a general case of harmonic transplantation. 
Also, we explain the difference between \tcr{two transformations.} 
In \S \ref{S Proof}, we show Theorem \ref{Thm CH} by \tcr{exploiting} the M\"obius transformation. 
Due to the lack of the explicit form of \tcr{$p$-}Green's function \tcr{$G_{\re^N_+, e_N}$} for $p \in (2,N)$, 
it seems difficult to apply the original harmonic transplantation which exploits the $p$-Green's functions,
to obtain an improvement of the Hardy inequality on the half-space $\re^N_+$, see Theorem \ref{Thm psi} in \S \ref{S Transformation}.
We use $U_p$ in \eqref{def U} instead of $p$-Green's function $G_{\re^N_{+}, e_N}$ to define a modified version of the harmonic transplantation.
By usng this new transformation, we show Theorem \ref{Thm IH}.
In the last of \S \ref{S Proof}, we mention that these transformations can be also applicable to Sobolev type inequalities. 
In \S \ref{S App}, we show several propositions related to main theorems and give \tcr{an} application of \tcr{a special type} of harmonic transplantation.


We fix several notations: 
$B_R$ or $B_R^N$ denotes the $N$-dimensional ball centered $0$ with radius $R$. As a matter of convenience, we set $B_\infty^N = \re^N$ and $\frac{1}{\infty} = 0$.  
$\omega_{N-1}$ denotes the area of the unit sphere $\mathbb{S}^{N-1}$ in $\re^N$. 
$[f > \ep]$ denotes the set $\{ (x,y) \in \re^N_+ \,|\, f(x,y) >\ep \}$. 
$|A|$ denotes the Lebesgue measure of a set $A \subset \re^N$.

%
%

\section{Green's function on the half-space}\label{S Green}

Let $G_{\Omega, a} = G_{\Omega, a} (z): \Omega \setminus \{ a\} \to \re$ be the \tcr{$p$-}Green function 
with singularity at $a \in \Omega$ associated with $p$-Laplace operator $\lap_p (\cdot)= {\rm div }(|\nabla (\cdot ) |^{p-2} \nabla (\cdot ))$.
Namely, $G_{\Omega, a} (z)$ satisfies 
\begin{align}
\label{lap_p}
\begin{cases}
	-\lap_p G_{\Omega, a}(z) = \delta_{a}(z), &\quad z \in \Omega,\\
	\quad \,\,\,\,  G_{\Omega, a} (z) = 0 \,\,&\quad z \in \pd \Omega,
\end{cases}
\end{align}
where $\delta_a$ is the Dirac measure giving unit mass to a point $a \in \Omega$. 
When $\Omega = \re^N_+$ and \tcr{$a= e_N = (0,1)$}, we have
\begin{align}
\label{G half}
	\tcr{G_{\re^N_+, e_N}(x,y)} = 
	\begin{cases}
	\frac{p-1}{N-p} \w_{N-1}^{-\frac{1}{p-1}} \left[ \( |x|^2 +(1-y)^2 \)^{-\frac{N-p}{2(p-1)}} - \psi_p (x, y)  \right] \, &\text{if} \, p \in (1,N), \\
	\w_{N-1}^{-\frac{1}{N-1}} \log \sqrt{ \frac{ |x|^2 +(1+y)^2}{ |x|^2 +(1-y)^2}} &\text{if}\, p=N,
	\end{cases}
\end{align}
where \tcr{$\psi_p$ is a function with} $\psi_p \in L^\infty_{\rm loc} (\re^N_+)$, 
$\lim_{|x|^2 + (y-1)^2 \to 0} \( |x|^2 +(1-y)^2 \)^{\frac{N-1}{2(p-1)}} \nabla \psi_p (x, y) = 0$, 
and $\psi_2 (x,y)= \( |x|^2 +(1+y)^2 \)^{-\frac{N-2}{2}}$ 
(see \cite{KV}). 
\tcr{Note} that $U_p$ \tcr{in \eqref{def U} coincides with} $G_{\re^N_+, e_N}$ for $p=2$ or $p=N$. 
To the best of our knowledge, we do not know the explicit form of $\psi_p$ when $p \not= 2$ and $p \not= N$. 
\tcr{This fact causes some difficulty in the application of harmonic transplantation in \S \ref{S Transformation} on $\re^N_{+}$, 
since we need the explicit form of Green's function in the use of harmonic transplantation.} 
However, fortunately, we see that $U_p$ is a super (or sub) solution of (\ref{lap_p}) in the distributional sense \tcr{according to the range of $p$} as follows. 
This fact enables us to use $U_p$ instead of $G_{\re^N_{+}, e_N}$ in the proof of Theorem \ref{Thm IH}.

%
%
\begin{prop}\label{Prop sol U}
Let $1 < p \le N$ \tcr{and let $U_p$ be as in \eqref{def U}.} 
Then for any $\phi \in C_c^\infty (\re^N_+)$ with $\phi \ge 0$, 
\begin{align}
\label{eq:sol U}
	\int_{\re^N_+} | \nabla U_p|^{p-2} \nabla U_p \cdot \nabla \phi \,dxdy 
	&= \phi (0,1) + \int_{\re^N_+ } (-\lap_p U)\, \phi \, dxdy \\
	&\begin{cases}
	\le \phi (0, 1) \quad &\text{if} \,\, p \in (1, 2], \\
	\ge \phi (0, 1)  &\text{if} \,\, p \in [2, N), \\
	= \phi (0, 1) &\text{if} \,\, p = N.
	\end{cases} \notag
\end{align} 
\end{prop}

Proposition \ref{Prop sol U} follows from Proposition \ref{Prop cal U}. 

\begin{prop}
\label{Prop cal U}
\tcr{Let $1 < p \le N$ and let $U_p$ be as in \eqref{def U}.} 
Then for $(x,y) \in \re^N_+ \setminus \{ (0,1) \}$, we have the followings:

(I) Let $1<p<N$. 
Then
\begin{align*}
	-\lap_p U_p 
	&=  \frac{(N-p)(p-2)}{(p-1)^2 \, \w_{N-1}^{\frac{2}{p-1}}} \, |\nabla U_p|^{p-4} U_p 
	\left[ |x|^2 + (y-1)^2 \right]^{-\frac{N-p}{2(p-1)} -1} \left[ |x|^2 + (y+1)^2 \right]^{-\frac{N-p}{2(p-1)} -1} \\
	&\hspace{3em}\times \left[ N-p + (N+p-2) \frac{(|x|^2+ y^2 -1)^2}{\{ |x|^2 + (y-1)^2 \} \{ |x|^2 + (y+1)^2 \} } \right].
\end{align*}

(II)  $-\lap_N U_N = 0$.

\tcr{Especially, we see that the pointwise estimates
$$
	\begin{cases}
	-\Delta_p U_p \le 0 &\quad \text{on} \ \re^N_{+} \setminus \{ e_N \}, \quad (1 < p \le 2), \\
	-\Delta_p U_p \ge 0 &\quad \text{on} \ \re^N_{+} \setminus \{ e_N \}, \quad (2 \le p < N), \\
	-\Delta_p U_p = 0 &\quad \text{on} \ \re^N_{+} \setminus \{ e_N \}, \quad (p = N)
	\end{cases}
$$
hold.}
\end{prop}

\begin{proof}(Proof of Proposition \ref{Prop cal U})

\noindent
(I) For $(x,y) \in \re^N_+ \setminus \{ (0,1) \}$, we have
\begin{align*}
	\w_{N-1}^{\frac{1}{p-1}} \nabla U_p = -\left[ |x|^2 + (y-1)^2 \right]^{-\frac{N-p}{2(p-1)}-1} 
	\begin{pmatrix}
	x \\
	y-1
	\end{pmatrix}
	+ \left[ |x|^2 + (y+1)^2 \right]^{-\frac{N-p}{2(p-1)} -1} 
	\begin{pmatrix}
	x \\
	y+1
	\end{pmatrix}
\end{align*}
which implies that
\begin{align*}
	|\nabla U_p|^2 
	&= \w_{N-1}^{-\frac{2}{p-1}} \Biggl[ \left\{|x|^2 + (y-1)^2 \right\}^{-\frac{N-p}{p-1}-1} 
	+ \left\{|x|^2 + (y+1)^2 \right\}^{-\frac{N-p}{p-1}-1} \\
	&\hspace{1em}- 2 \left\{|x|^2 + (y-1)^2 \right\}^{-\frac{N-p}{2(p-1)}-1} \left\{|x|^2 + (y+1)^2 \right\}^{-\frac{N-p}{2(p-1)}-1} 
	\{ |x|^2 +y^2 -1 \} \Biggl].
\end{align*}
\tcr{We put $V = |\nabla U_p|^2$.}
Then we have
\begin{align*}
	&{\rm div}(|\nabla U_p |^{p-2} \nabla U_p) 
	= {\rm div}(V^{\frac{p-2}{2}} \nabla U_p)
	= V^{\frac{p-4}{2}} \left[ V \lap U_p + \frac{p-2}{2} \nabla V \cdot \nabla U_p  \right], \\
	&\lap U_p = {\rm div} (\nabla U_p) 
	= -\frac{(N-1)(p-2)}{(p-1)\, \w_{N-1}^{\frac{1}{p-1}}} 
	\left[ \left\{|x|^2 + (y-1)^2 \right\}^{-\frac{N-p}{2(p-1)}-1} - \left\{|x|^2 + (y+1)^2 \right\}^{-\frac{N-p}{2(p-1)}-1}  \right].
\end{align*}
Also, we have
\begin{align*}
	&\w_{N-1}^{\frac{2}{p-1}} \nabla V \\ 
	&=-2 \,\frac{N-1}{p-1} \left\{ |x|^2 + (y-1)^2 \right\}^{-\frac{N-p}{p-1}-2}
	\begin{pmatrix}
	x \\
	y-1
	\end{pmatrix}
	- 2 \, \frac{N-1}{p-1} \left\{|x|^2 + (y+1)^2 \right\}^{-\frac{N-p}{p-1}-2}
	\begin{pmatrix}
	x \\
	y+1
	\end{pmatrix} \\
	&-4 \left\{|x|^2 + (y-1)^2 \right\}^{-\frac{N-p}{2(p-1)}-1} \left\{|x|^2 + (y+1)^2 \right\}^{-\frac{N-p}{2(p-1)}-1}
	\begin{pmatrix}
	x \\
	y
	\end{pmatrix} \\
	&+2 \, \( \frac{N-p}{p-1}+2 \)  \left\{|x|^2 + (y-1)^2 \right\}^{-\frac{N-p}{2(p-1)}-2} 
	\left\{|x|^2 + (y+1)^2 \right\}^{-\frac{N-p}{2(p-1)}-1} (|x|^2 +y^2 -1)
	\begin{pmatrix}
	x \\
	y-1
	\end{pmatrix} \\
	&+2 \, \( \frac{N-p}{p-1}+2 \) \left\{  |x|^2 + (y-1)^2 \right\}^{-\frac{N-p}{2(p-1)}-1} 
	\left\{|x|^2 + (y+1)^2 \right\}^{-\frac{N-p}{2(p-1)}-2} (|x|^2 +y^2-1)
	\begin{pmatrix}
	x \\
	y+1
	\end{pmatrix}
\end{align*}
which implies that
\begin{align*}
	&\frac{p-2}{2} \nabla V \cdot \nabla U_p \\
	&= \frac{p-2}{2} \w_{N-1}^{-\frac{3}{p-1}} \Biggl[\left[ |x|^2 + (y-1)^2 \right]^{-\frac{N-p}{2(p-1)}-1}
	\begin{pmatrix}
	x \\
	y-1
	\end{pmatrix}
	- \left[ |x|^2 + (y+1)^2 \right]^{-\frac{N-p}{2(p-1)} -1}
	\begin{pmatrix}
	x \\
	y+1
	\end{pmatrix} \Biggl] \cdot \\
	&\Biggl[2 \,\frac{N-1}{p-1} \left\{  |x|^2 + (y-1)^2 \right\}^{-\frac{N-p}{p-1} -2}
	\begin{pmatrix}
	x \\
	y-1
	\end{pmatrix} 
	+ 2 \, \frac{N-1}{p-1} \left\{|x|^2 + (y+1)^2 \right\}^{-\frac{N-p}{p-1}-2}
	\begin{pmatrix}
	x \\
	y+1
	\end{pmatrix} \\
	&+4 \left\{|x|^2 + (y-1)^2 \right\}^{-\frac{N-p}{2(p-1)}-1} \left\{ x|^2 + (y+1)^2 \right\}^{-\frac{N-p}{2(p-1)}-1}
	\begin{pmatrix}
	x \\
	y
	\end{pmatrix} \\
	&-2 \, \( \frac{N-p}{p-1} +2 \)  \left\{|x|^2 + (y-1)^2 \right\}^{-\frac{N-p}{2(p-1)}-2} 
	\left\{  |x|^2 + (y+1)^2 \right\}^{-\frac{N-p}{2(p-1)} -1} (|x|^2 +y^2 -1)
	\begin{pmatrix}
	x \\
	y-1
	\end{pmatrix} \\
	&-2 \, \( \frac{N-p}{p-1} +2 \) \left\{|x|^2 + (y-1)^2 \right\}^{-\frac{N-p}{2(p-1)}-1} 
	\left\{|x|^2 + (y+1)^2 \right\}^{-\frac{N-p}{2(p-1)}-2} (|x|^2 +y^2 -1)
	\begin{pmatrix}
	x \\
	y+1
	\end{pmatrix} \Biggl].
\end{align*}
Therefore, we have
\begin{align*}
	&\frac{p-2}{2} \nabla V \cdot \nabla U_p \\
	&=\frac{p-2}{2} \w_{N-1}^{-\frac{3}{p-1}} \Biggl[ 2 \,\frac{N-1}{p-1}
	\left[ |x|^2 + (y-1)^2 \right]^{-\frac{3(N-p)}{2(p-1)} -2} \\
	&+2 \,\frac{N-1}{p-1}
	\left[ |x|^2 + (y-1)^2 \right]^{-\frac{N-p}{2(p-1)} -1} \left[ |x|^2 + (y+1)^2 \right]^{-\frac{N-p}{p-1} -2} (|x|^2 +y^2 -1) \\
	&+4 \left[ |x|^2 + (y-1)^2 \right]^{-\frac{N-p}{p-1} -2} \left[ |x|^2 + (y+1)^2 \right]^{-\frac{N-p}{2(p-1)} -1} (|x|^2 +y^2 -y) \\
	&-2 \, \( \frac{N-1}{p-1} +1 \)  \left\{  |x|^2 + (y-1)^2 \right\}^{-\frac{N-p}{p-1} -2} \left\{  |x|^2 + (y+1)^2 \right\}^{-\frac{N-p}{2(p-1)} -1} (|x|^2 +y^2 -1) \\
	&-2 \, \( \frac{N-1}{p-1} +1 \)  \left\{  |x|^2 + (y-1)^2 \right\}^{-\frac{N-p}{p-1} -2} \left\{  |x|^2 + (y+1)^2 \right\}^{-\frac{N-p}{2(p-1)} -2} (|x|^2 +y^2 -1)^2 \\
	&-2 \,  \frac{N-1}{p-1} \left\{  |x|^2 + (y-1)^2 \right\}^{-\frac{N-p}{p-1} -2} \left\{  |x|^2 + (y+1)^2 \right\}^{-\frac{N-p}{2(p-1)} -1} (|x|^2 +y^2 -1) \\
	&-2 \,\frac{N-1}{p-1}
	\left[ |x|^2 + (y+1)^2 \right]^{-\frac{3(N-p)}{2(p-1)} -2}
	+4 \left[ |x|^2 + (y-1)^2 \right]^{-\frac{N-p}{2(p-1)} -1} \left[ |x|^2 + (y+1)^2 \right]^{-\frac{N-p}{2(p-1)} -2} (|x|^2 +y^2 +y) \\
	&+2 \, \( \frac{N-1}{p-1} +1 \)  \left\{  |x|^2 + (y-1)^2 \right\}^{-\frac{N-p}{2(p-1)} -2} \left\{  |x|^2 + (y+1)^2 \right\}^{-\frac{N-p}{p-1} -2} (|x|^2 +y^2 -1)^2 \\
	&+2 \, \( \frac{N-1}{p-1} +1 \)  \left\{  |x|^2 + (y-1)^2 \right\}^{-\frac{N-p}{2(p-1)} -1} \left\{  |x|^2 + (y+1)^2 \right\}^{-\frac{N-p}{p-1} -2} (|x|^2 +y^2 -1)
	\Biggl].
\end{align*}
Since
\begin{align*}
	V \lap U_p = - \w_{N-1}^{-\frac{3}{p-1}} \frac{p-2}{2} \,
	&\Biggl[ 2 \,\frac{N-1}{p-1} \left\{  |x|^2 + (y-1)^2 \right\}^{-\frac{3(N-p)}{2(p-1)}-2} \\
	&- 2 \,\frac{N-1}{p-1} \left\{|x|^2 + (y-1)^2 \right\}^{-\frac{N-p}{p-1}-1} \left\{|x|^2 + (y+1)^2 \right\}^{-\frac{N-p}{2(p-1)}-1} \\
	&+ 2 \,\frac{N-1}{p-1} \left\{|x|^2 + (y-1)^2 \right\}^{-\frac{N-p}{2(p-1)}-1} \left\{|x|^2 + (y+1)^2 \right\}^{-\frac{N-p}{p-1}-1} \\
	&- 2 \,\frac{N-1}{p-1} \left\{|x|^2 + (y+1)^2 \right\}^{-\frac{3(N-p)}{2(p-1)}-2} \\
	&-4 \,\frac{N-1}{p-1} \left\{|x|^2 + (y-1)^2 \right\}^{-\frac{N-p}{p-1} -2} 
	\left\{ x|^2 + (y+1)^2 \right\}^{-\frac{N-p}{2(p-1)}-1} (|x|^2 +y^2 -1) \\
	&+4 \,\frac{N-1}{p-1} \left\{  |x|^2 + (y-1)^2 \right\}^{-\frac{N-p}{2(p-1)}-1} 
	\left\{|x|^2 + (y+1)^2 \right\}^{-\frac{N-p}{p-1}-2} (|x|^2 + y^2 - 1) \Biggl],
\end{align*}
we have
\begin{align*}
	&{\rm div}(|\nabla U_p |^{p-2} \nabla U_p) V^{-\frac{p-4}{2}} \w_{N-1}^{\frac{3}{p-1}} \frac{2}{p-2}
	= \left[ V \lap U_p + \frac{p-2}{2} \nabla V \cdot \nabla U_p  \right] \w_{N-1}^{\frac{3}{p-1}} \frac{2}{p-2} \\
	&= -2 \,\frac{N-1}{p-1} \left\{|x|^2 + (y-1)^2 \right\}^{-\frac{N-p}{2(p-1)}-1} 
	\left\{|x|^2 + (y+1)^2 \right\}^{-\frac{N-p}{2(p-1)}-1} U_p \frac{N-p}{p-1} \, \w_{N-1}^{\frac{1}{p-1}} \\
	&+2 \left\{|x|^2 + (y-1)^2 \right\}^{-\frac{N-p}{2(p-1)}-1} \left\{|x|^2 + (y+1)^2 \right\}^{-\frac{N-p}{2(p-1)}-1} 
	\Biggl[ \left\{|x|^2 + (y+1)^2 \right\}^{-\frac{N-p}{2(p-1)}-1} (|x|^2 +y^2 -1) \\
	&- \left\{|x|^2 + (y-1)^2 \right\}^{-\frac{N-p}{2(p-1)}-1} (|x|^2 +y^2 -1)
	+2 \left\{|x|^2 + (y-1)^2 \right\}^{-\frac{N-p}{2(p-1)}-1} (|x|^2 +y^2 -y) \\
	&- 2 \left\{|x|^2 + (y+1)^2 \right\}^{-\frac{N-p}{2(p-1)}-1} (|x|^2 +y^2 +y)
	\Biggl] \\
	&- 2 \( \frac{N-1}{p-1} +1 \) \left\{  |x|^2 + (y-1)^2 \right\}^{-\frac{N-p}{2(p-1)} -2} 
	\left\{  |x|^2 + (y+1)^2 \right\}^{-\frac{N-p}{2(p-1)} -2} (|x|^2 +y^2 -1)^2 U_p\\
	&=- \frac{2(N-p) \, \w_{N-1}^{\frac{1}{p-1}}}{(p-1)^2} U_p 
	\left\{|x|^2 + (y-1)^2 \right\}^{-\frac{N-p}{2(p-1)} -1} \left\{  |x|^2 + (y+1)^2 \right\}^{-\frac{N-p}{2(p-1)} -1} \\
	&\left[ N-p + (N+p-2) \frac{(|x|^2+ y^2 -1)^2}{\{ |x|^2 + (y-1)^2 \} \{ |x|^2 + (y+1)^2 \} } \right]
\end{align*}
\tcr{which implies Proposition \ref{Prop cal U} (I).}


\noindent
(II) The proof is done by direct calculation in the same way as (I). We omit the proof here.

\end{proof}


\begin{proof}(Proof of Proposition \ref{Prop sol U})
\tcr{Let $B_\ep(e_N)$ be the ball with center $e_N$ and radius $\ep$.}
For any $\phi \in C_c^{\infty} (\re^N_+)$, we have
\begin{align}
	&\int_{\re^N_+} | \nabla U_p |^{p-2} \nabla U_p \cdot \nabla \phi \, dxdy \notag \\
\label{PU0}
	&= \int_{\pd B_\ep(e_N)} | \nabla U_p |^{p-2} ( \nabla U_p \cdot \nu )\,\phi \, dS
	+ \int_{\re^N_+ \setminus B_\ep(e_N)} (-\lap_p U_p)\, \phi \, dxdy 
	+ \int_{B_\ep(e_N)} | \nabla U_p |^{p-2} \nabla U_p \cdot \nabla \phi \, dxdy.
\end{align}
where $\nu = -(x, y-1)^{T} \left\{ |x|^2 + (y-1)^2 \right\}^{-\frac{1}{2}}$. 
\tcr{We claim}
\begin{align}\label{1st term}
 \int_{\pd B_\ep(e_N)} | \nabla U_p |^{p-2} ( \nabla U_p \cdot \nu )\,\phi \, dS
&= \phi (0,1) + o(1)\quad (\ep \to 0), \\
\label{3rd term}
\int_{B_\ep(e_N)} | \nabla U_p |^{p-2} \nabla U_p \cdot \nabla \phi \, dxdy
&= o(1) \quad (\ep \to 0),
\end{align}
Indeed, a direct calculation shows that
\begin{align*}
	&\int_{\pd B_\ep(e_N)} | \nabla U_p |^{p-2} ( \nabla U_p \cdot \nu )\,\phi \, dS\\
	&= \w_{N-1}^{-1} \int_{\pd B_\ep} \Biggl[ \left\{|x|^2 + (y-1)^2 \right\}^{-\frac{N-p}{p-1}-1} 
	+ \left\{|x|^2 + (y+1)^2 \right\}^{-\frac{N-p}{p-1}-1} 
	- 2 \left\{  |x|^2 + (y-1)^2 \right\}^{-\frac{N-p}{2(p-1)}-1} \\
	&\left\{|x|^2 + (y+1)^2 \right\}^{-\frac{N-p}{2(p-1)}-1} \{ |x|^2 +y^2 -1 \}  \Biggl]^{\frac{p-2}{2}} 
	\Biggl[ \left\{|x|^2 + (y-1)^2 \right\}^{-\frac{N-p}{2(p-1)} } - \left\{|x|^2 + (y+1)^2 \right\}^{-\frac{N-p}{2(p-1)}-1} \\
	&\{ |x|^2 +y^2 -1 \} \Biggl] \left\{ |x|^2 + (y-1)^2 \right\}^{-\frac{1}{2}} \,\phi \, dS\\
	&= \w_{N-1}^{-1} \ep^{-1-\frac{N-p + (N-1)(p-2)}{p-1}} 
	\int_{\pd B_\ep(e_N)} \Biggl[1 + \left\{  \frac{\ep^2}{|x|^2 + (y+1)^2} \right\}^{\frac{N-1}{p-1}} 
	- 2 \ep^{-\frac{N-1}{p-1}-1} \left\{ |x|^2 + (y+1)^2 \right\}^{-\frac{N-p}{2(p-1)}-1}\\
	&\{ |x|^2 +y^2 -1 \}  \Biggl]^{\frac{p-2}{2}} 
	\Biggl[1 - \left\{ \frac{\ep^2}{|x|^2 + (y+1)^2} \right\}^{\frac{N-p}{2(p-1)}} \frac{|x|^2 +y^2 -1}{|x|^2 +(y+1)^2}  \Biggl]  \,\phi \, dS\\
	&= \phi (0,1) + o(1) \quad (\ep \to 0)
\end{align*}
which implies (\ref{1st term}). 

On the other hand, we \tcr{see} 
\begin{align*}
	\left| \, \int_{B_\ep(e_N)} | \nabla U_p |^{p-2} \nabla U_p \cdot \nabla \phi \, dxdy\, \right|
	&\le C \, \| \nabla \phi \|_\infty \int_{B_\ep(e_N)} \left\{ |x|^2 + (y-1)^2 \right\}^{-\frac{N-1}{2}} \,dxdy \\
	&= C \, \| \nabla \phi \|_\infty \, \w_{N-1} \int_0^\ep \, dr = o(1) \quad (\ep \to 0),
\end{align*}
\tcr{which proves \eqref{3rd term}.} 

\tcr{
Finally, we check that the second term in \eqref{PU0} satisfies
\begin{equation}
\label{2nd term}
	\int_{\re^N_+ \setminus B_\ep(e_N)} (-\lap_p U_p)\, \phi \, dxdy \to \int_{\re^N_+ } (-\lap_p U_p)\, \phi \, dxdy
\end{equation}
as $\ep \to 0$.
Actually, we have 
\begin{align*}
	|\Delta_p U_p| &= O\( |\nabla U_p|^{p-4} U_p (|x|^2 + (y-1)^2 )^{-\frac{N-p}{2(p-1)}-1} \), \\ 
	|\nabla U_p|^{p-4} &= O\( (|x|^2 + (y-1)^2 )^{(-\frac{N-p}{p-1}-1)(\frac{p-4}{2})} \), \\ 
	|U_p| &= O\( (|x|^2 + (y-1)^2 )^{-\frac{N-p}{2(p-1)}} \)
\end{align*}
near $(x, y) = (0, 1)$ by Proposition \ref{Prop cal U}.
Thus we have
\begin{align*}
	|\Delta_p U_p| &= O\( (|x|^2 + (y-1)^2 )^{(-\frac{N-p}{p-1}-1)(\frac{p-4}{2}) - \frac{N-p}{2(p-1)} - \frac{N-p}{2(p-1)}-1} \) \\
	&= O\( \(\sqrt{|x|^2 + (y-1)^2}\)^{(-\frac{N-p}{p-1}-1)(p-4) - \frac{N-p}{(p-1)} - \frac{N-p}{(p-1)}-2} \) \\
	&= O\( \(\sqrt{|x|^2 + (y-1)^2}\)^{-\frac{(N-1)(p-2)}{p-1}} \) 
\end{align*}
near $(x, y) = (0, 1)$.
This is locally integrable if
$$
	\frac{(N-1)(p-2)}{p-1} < N
$$
which always holds for $p \in (1, N]$.
Thus \eqref{2nd term} follows from Lebesgue's dominated convergence theorem.
}

\tcr{Returning to \eqref{PU0} with \eqref{1st term}, \eqref{3rd term}, and \eqref{2nd term},
we obtain \eqref{eq:sol U}.}
\end{proof}

%
%

\section{M\"obius transformation and harmonic transplantation}\label{S Transformation}

In this section, we recall M\"obius transformation and harmonic transplantation. 
Both transformations preserve the norm $\| \nabla (\cdot) \|_p$ and coincide in the critical case $p=N$. 
However in the subcritical case $p<N$, these transformations are different \tcr{from} each other. 


\subsection{M\"obius transformation}\label{S Mobius}

First, we recall \tcr{the definition of} M\"obius transformation and its properties. 

\begin{definition}\label{def M}(M\"obius transformation)
For $b \in \re^N$, $\la >0$, $R \in O(N)$, where $O(N)$ is the orthogonal group in \tcr{$\re^N$}, set
\begin{align*}
T_b (z) &= z+b\quad (\text{translation}),\\
S_\la (z) &= \la z \quad (\text{scaling}),\\
R(z) &= Rz \quad (\text{rotation}),\\
J(z) &= z^* = \frac{z}{|z|^2} \quad (\text{reflection}).
\end{align*}
A {\it M\"obius transformation} $M: \re^N \to \re^N$ is a finite composition of $T_b, S_\la, R$ and $J$. 
Also, the group of M\"obius transformations is denoted by $M(\re^N)$. 
\end{definition}

\begin{remark}\label{Jacobian}
Set $u \otimes v = (u_i v_j)_{1 \le i, j \le N}$ for $u=(u_1, \cdots, u_N)^T$ and $v=(v_1, \cdots, v_N)^T$. 
The \tcr{differential} and the Jacobian of each transformation are as follows.
\begin{align*}
(T_b)' (z) &= I, \,\,\det (T_b)'(z) =1\\
(S_\la)' (z) &= \la I, \,\, \det (S_\la)'(z) =\la^N\\
R'(z) &= R, \,\, \det R' (z) = \det R = \pm 1\\
J'(z) &= \frac{1}{|z|^2}\( I - 2 \frac{z}{|z|} \otimes \frac{z}{|z|}  \), \,\, \det J'(z) = \frac{(-1)}{|z|^{2N}}
\end{align*}
\tcr{where $I$ is the identity matrix on $\re^N$.}
For the proof of the last one, see Proposition \ref{Prop detJ}.
\end{remark}

\tcr{For a function $f: \re^N \to \re$, we set}
\begin{align}\label{M trans}
(M^{\#} f ) (z) = |\det M' (z)|^{\frac{N-p}{Np}} f(M(z))
\end{align}
for $z \in \re^N$. 
\tcr{We call $M^{\sharp} f$ is also the M\"obius transformation of the function $f$.}
Then we see that the transformation $M^{\#}$ preserves several quantities as follows.

\begin{prop}\label{Prop Mobius}
Let $1\le p \le N$ and $p^* = \frac{Np}{N-p}$ for $p < N$. 
If $M \in M(\re^N)$, then 
\begin{align}\label{nabla M}
\int_{\re^N} |\nabla (M^{\#} f) (z) |^p\,dz &= \int_{\re^N} |\nabla f (w)|^p \,dw \quad \text{for} \,\, p=2 \,\text{or} \,N,\\
\label{p^* M}
\int_{\re^N} |(M^{\#} f) (z) |^{p^*}\,dz &= \int_{\re^N} |f (w)|^{p^*} \,dw \quad \text{for} \,\, p <N, \\
\label{Hardy M}
\int_{\re^N} \frac{|(M^{\#} f) (z) |^{p}}{|z|^p}\,dz &= \int_{\re^N} \frac{|f (w)|^{p}}{|w|^p} \,dw 
\end{align}
hold for any $f \in C_c^1 (\re^N \setminus \{ 0 \})$. 
\end{prop}

\begin{remark}\label{Kelvin}
From Remark \ref{Jacobian}, we have 
\begin{align*}
(T_b^{\#} \,f) (z) &= f(z+b),\\
(S_\la^{\#}\, f) (z) &= \la^{\frac{N-p}{p}} f(\la z),\\
(R^{\#} f) (z) &= f(Rz),\\
(J^{\#} f )(z) &= |z|^{\frac{2}{p} (p-N)} f \( \frac{z}{|z|^2} \).
\end{align*}
The last transformation is called \tcr{the} Kelvin transformation when $p=2$ or $N$. 
In the case where $p \not=2$ and $p \not= N$, there is no radial function $\rho= \rho (|z|) \not\equiv 0$ such that 
\begin{align*}
\int_{\re^N} |\nabla g(z) |^p\,dz = \int_{\re^N} |\nabla f (w)|^p \,dw \,\,\text{holds for}\,\, g(z) = \rho (|z|) f \( \frac{z}{|z|^2} \),
\end{align*}
see the proof below. \tcb{Furthermore, all transformations except for $J^{\#}$ above preserve the $p$-harmonicity of functions: $\Delta_p f = 0$ implies $\Delta_p (M^{\#} f) = 0$, where $M$ is one of $T_b$, $S_{\la}$, and $R$. Also $J^{\#}$ preserves the $p-$harmonicity of functions when $p=2$ or $p=N$. }
\tcr{When $p \ne 2$ and $p \ne N$, it is shown in \cite{Lind} that there is no radial function $\rho$ such that $\Delta_p g = 0$, $g$ as above, for any function $f$ satisfying $\Delta_p f = 0$.}
\end{remark}

\begin{proof} \tcr{(Proof of Proposition \ref{Prop Mobius})}
We can easily show (\ref{p^* M}) and (\ref{Hardy M}). We show (\ref{nabla M}) only.  
First, we claim that \tcr{\eqref{nabla M} holds for each transformation $T_b, S_\la, R, J$.}
We shall show (\ref{nabla M}) only \tcr{for} $M=J$. 
We use the polar coordinate $z= r\w, r=|z|, \w \in \mathbb{S}^{N-1}$. 
Then we have $w := Jz = s \omega$, $|w| = s = r^{-1}$, and
\begin{align*}
(J^{\#}f )(r\w) =\rho (r) f (s \w), \quad \text{where} 
\,\, \rho (r)= r^{\frac{2}{p} (p-N)}.
\end{align*}
Therefore, we have
\begin{align*}
&\int_{\re^N} |\nabla (J^{\#} f) (z) |^p\,dz \\
&= \int_0^\infty \int_{\mathbb{S}^{N-1} } \left[ \left| \,\frac{\pd (J^{\#}f )}{\pd r} \,\right|^2  + \frac{1}{r^2} |\nabla_{\mathbb{S}^{N-1}} (J^{\#}f) |^2  \right]^{\frac{p}{2}} r^{N-1} \,dr dS_{\w} \\
&= \iint \left[ \left| -\frac{\pd f}{\pd s} \rho (r) r^{-2} + \rho' (r) f \,\right|^2  + \tcr{\rho (r)^2} r^{-2} |\nabla_{\mathbb{S}^{N-1}} f |^2  \right]^{\frac{p}{2}} r^{N-1} \,dr dS_{\w} \\
&= \iint \left[ \left| \frac{\pd f}{\pd s} - \frac{\rho' (r) r^2}{\rho (r)} f \,\right|^2  + r^2 |\nabla_{\mathbb{S}^{N-1}} f |^2  \right]^{\frac{p}{2}} \rho(r)^{p} r^{-2p + N-1} \,dr dS_{\w} \\
&= \iint \left[ \left| \,\frac{\pd f}{\pd s} \, \right|^2 -\frac{\rho'(r)}{\rho (r) s^2}  \frac{\pd}{\pd s} (f^2) + \frac{|\rho'(r)|^2}{\rho (r)^2 s^4}  f^2  + \frac{1}{s^2} |\nabla_{\mathbb{S}^{N-1}} f |^2  \right]^{\frac{p}{2}} \rho(r)^p r^{2(N-p)} s^{N-1}\,ds dS_{\w}
\end{align*}
\tcr{where $r = s^{-1}$.}
Since $\rho (r)^p r^{2(N-p)} =1$ and $\rho (r) =1$ for $p=N$, we obtain (\ref{nabla M}) for $p=N$. 
In the case where $p=2$, by the integration by parts, we have
\begin{align*}
\int_{\re^N} |\nabla (J^{\#} f) (z) |^2\,dz 
&= \int_{\re^N} |\nabla f (w) |^2 \,dw + \tcr{(N-2)^2} \iint \( \frac{\pd}{\pd s} \( \frac{1}{s} \) f^2  + \frac{1}{s^2}  f^2 \) s^{N-1}\,ds dS_{\w}\\
&= \int_{\re^N} |\nabla f (w) |^2 \,dw.
\end{align*}
Therefore, we obtain the claim. Let $A, B \in \{ T_b, S_\la, R, J\}$. 
Since
\begin{align*}
(A \circ B)^{\#} (z) 
&=|\det (A \circ B)' (z)|^{\frac{N-p}{Np}} f((A \circ B)(z)) \\
&=|\det A' ( B (z) ) \cdot \det B'(z)|^{\frac{N-p}{Np}} f(A (B(z) ) ) \\
&=|\det B'(z)|^{\frac{N-p}{Np}} (A^{\#}f )(B(z) ) \\
&= [(B^{\#} \circ A^{\#}) f] (z),
\end{align*}
we have
\begin{align*}
(A \circ B)^{\#} = B^{\#} \circ A^{\#}.
\end{align*}
From this and the claim, we have
\begin{align*}
\int_{\re^N} |\nabla [(A \circ B)^{\#} f ](z) |^p\,dz 
&=\int_{\re^N} |\nabla [ (B^{\#} \circ A^{\#} ) f ](z) |^p\,dz  \\
&=\int_{\re^N} |\nabla (A^{\#}  f )(w) |^p\,dw  \\
&= \int_{\re^N} |\nabla f (\xi)|^p \,d\xi
\end{align*}
for $p=2$ or $N$, where $w=B(z)$ and $\xi = A(w) = (A \circ B) (z)$. 
Since M\"obius transformation is a finite composition of $T_b, S_\la, R, J$, we obtain (\ref{nabla M}) for any $M \in M(\re^N)$ by induction. 
\end{proof}

For more information \tcr{on} M\"obius transformation, see e.g. \cite{A,B}.


\subsection{An example of M\"obius transformation: Cayley type transformation}\label{S Cayley}

Let $N \ge 2$ and $p=2$ or $N$. 
Consider the transformation ${\bf B}$ 
\tcr{from $\re^N$ to $\re^N$ as follows} (Ref. \cite{BFL}).
\begin{align}
\label{Cayley type}
	\( \tx, \ty \) = {\bf B} (x, y)= \( \frac{2x, \,\, 1-|x|^2 -y^2}{(1+y)^2 + |x|^2 }  \) \quad (x, y) \in \re^N, \,\, (\tx, \ty) \in \re^N.
\end{align}
\tcr{We see
\begin{equation}
\label{B maps to ball}
	|{\bf B}(x,y)|^2 = \frac{|x|^2 + (y-1)^2}{|x|^2 + (y +1)^2}, 
\end{equation}
thus if we restrict ${\bf B}$ on $\re^N_{+}$, then ${\bf B}$ maps $\re^N_{+}$ to the unit ball $B_1^N \subset \re^N$.
Also $|{\bf B}(x,y)| = 1$ if and only if $y = 0$, thus ${\bf B} (\pd \re^N_{+}) = \pd B_1^N$.}
We can check that the inverse function ${\bf B}^{-1}$ is the same as ${\bf B}$, that is 
$$
\( x, y \) = \tcr{{\bf B}^{-1}} (\tx, \ty)= \( \frac{2\tx, \,\, 1-|\tx|^2 -\ty^2}{(1+\ty)^2 + |\tx|^2 }  \).
$$
Note that the transformation (\ref{Cayley type}) is a M\"obius transformation\tcr{:${\bf B} \in M (\re^N)$.} 
In fact, we see that
\begin{align}\label{Cayley is Mobius}
{\bf B} (z) = R \circ J \circ T_{e_N} \circ S_2 \circ J \circ T_{-e_N} (z), \text{where}\,\,
R= 
\begin{pmatrix}
1 & & & \\
& \ddots & & \\
& & 1 & \\
& & & -1
\end{pmatrix},\,
e_N = \begin{pmatrix}
0 \\
\vdots \\
0 \\
1
\end{pmatrix},
\end{align}
see \cite{A} p.34 or Proposition \ref{Prop Cayley is Mobius} in \S \ref{S App}. 
Therefore, we obtain the following.

\begin{prop}\label{Prop Cayley}
Let ${\bf B}: \re^N \to B_1^N$ be given by (\ref{Cayley type}) and let $z=(x,y)$. Then
\begin{align*}
\det \tcr{{\bf B'}} (z) = -\left\{ \, \frac{2}{(1+y)^2 + |x|^2} \,\right\}^N. 
\end{align*}
\end{prop}

\begin{proof}
From (\ref{Cayley is Mobius}) and Remark \ref{Jacobian}, we have
 \begin{align*}
	\det {\bf B}'(z) &= \underbrace{\det R}_{=-1} \cdot \det J'\( T_{e_N} \circ S_2 \circ J \circ T_{-e_N} (z) \) \cdot \underbrace{\det (T_{e_N})' \( S_2 \circ J \circ T_{-e_N} (z) \)}_{=1} \\
	&\cdot \underbrace{\det S_2'\(J \circ T_{-e_N}(z) \)}_{=2^N} \cdot \det J'(T_{-e_N}(z)) \cdot \det \underbrace{(T_{-e_N})'(z)}_{=1} \\ 
	&= (-1) \cdot \frac{(-1)}{|T_{e_N} \circ S_2 \circ J \circ T_{-e_N}(z)|^{2N}} \cdot 2^N \cdot \frac{(-1)}{|T_{-e_N} (z)|^{2N}} \\
	&= (-1) \frac{1}{|e_N + 2(z-e_N)^*|^{2N}} \cdot \frac{1}{|z-e_N|^{2N}} \cdot 2^N \\
	&\tcr{= (-1) \frac{1}{|e_N + 2\frac{(z-e_N)}{|z-e_N|^2}|^{2N}} \cdot \frac{1}{|z-e_N|^{2N}} \cdot 2^N}
	= (-1) \frac{2^N}{\left\{ |x|^2 + (1 + y)^2 \right\}^N}. 
\end{align*}
\end{proof}


\subsection{Harmonic transplantation}\label{S Harmonic}


Harmonic transplantation \tcr{was first} proposed by \tcr{J.} Hersch \cite{H} 
\tcr{in the attempt to extend several isoperimetric problems on two-dimensional simply-connected domains to higher connectivity and higher dimensions,}
see also \cite{F book,BBF}. 
Here, we recall the original harmonic transplantation from $B_1^N$ to $\Omega \subset \re^N$.

For $v \in \dot{W}_{0, {\rm rad}}^{1,p}(B_1^N)$ and $a \in \Omega$, 
define $\tcr{H_a (v)} =u : \Omega \setminus \{ a\} \to \re$ by 
\begin{align}\label{Harmonic trans}
	u(y) = \tcr{H_a (v)}(y) 
	= v \( \, \( G_{B_1^N, O}  \)^{-1} \( G_{\Omega, a} (y) \)  \,\),
\end{align}
\tcr{where $G_{B_1^N, O}$ and $G_{\Omega, a}$ are $p$-Green's functions on the ball $B_1^N$ with the pole $O$ 
and on $\Omega$ with the pole $a \in \Omega$, respectively.}
In the case $p \in (1, N)$, we have 
\begin{equation}
\label{GB0}
	G_{B_1^N, O} (z) = \frac{p-1}{N-p} \w_{N-1}^{-\frac{1}{p-1}} \left[ \,  |z|^{-\frac{N-p}{p-1}}  - 1  \,  \right]
\end{equation}
which implies that
\begin{align*}
	\( G_{B_1^N, O}  \)^{-1} \( G_{\Omega, a} (y) \) 
	= \left[ \, \frac{N-p}{p-1} \w_{N-1}^{\frac{1}{p-1}} G_{\Omega, a} (y) +1  \, \right]^{-\frac{p-1}{N-p}}.
\end{align*}
Also, we can rewrite the transformation (\ref{Harmonic trans}) to
\begin{align}\label{Omega trans}
	u(y) = v(z), \,\text{where}\,\, G_{\Omega, a} (y) = G_{B_1^N, O} (z).
\end{align}
\tcr{Hereafter, we call the transformed function $u = H_a(v)$ on $\Omega \setminus \{ a \}$ via (\ref{Omega trans}) 
{\it the harmonic transplantation} of a function $v \in \dot{W}^{1,p}_{0, rad}(B^N_1)$.} 
\tcr{We see} that harmonic transplantation (\ref{Omega trans}) preserves $\| \nabla (\cdot)\|_{L^p}$ for \tcr{$p \in (1, N]$}. 
\tcr{A proof of this fact is shown for the sake of reader's convenience.}

\begin{lemma}\label{Lemma Omega}(\cite{F book} Theorem 10.3, \cite{CRN} Lemma 22)
Let $v \in \dot{W}_{0, {\rm rad}}^{1,p}(B_1^N)$ and $1<p \le N$. 
Then $\tcr{H_a(v)} \in \dot{W}_0^{1,p}(\Omega)$ and $\| \nabla \( \tcr{H_a(v)} \) \|_{L^p (\Omega)} = \| \nabla v \|_{L^p (B_1^N)}$.
\end{lemma}


\begin{proof}

In the case where $p=N$, see \cite{CRN} Lemma 22. 
Let $1<p<N$. 
We write $G=G_{\Omega, a}$. Let $h$ be defined by 
$$
	h(y) = \tcr{\( G_{B^N_1, O} \)^{-1} \( G(y) \)} = \left[ \, \frac{N-p}{p-1} \w_{N-1}^{\frac{1}{p-1}} G (y) +1  \, \right]^{-\frac{p-1}{N-p}},
	\quad \tcr{y \in \Omega,}
$$
and hence $u(y)= v(h(y))$. 
In particular, $\nabla u(y)= v' (h(y)) \nabla h(y)$. 
Note that since $G \ge 0$ in $\Omega$, we get that \tcr{$0 < h(y) \le 1$ on $\ol{\Omega}$ and if} $y \in h^{-1}(\{ t\} ) \cap \Omega$, 
then $t \in [0,1]$. 
Thus the coarea formula gives that 
\begin{align*}
	\int_{\Omega} |\nabla u |^p 
	&= \int_{\Omega} |v'(h(y))|^p |\nabla h (y)|^{p-1} |\nabla h (y)| \,dy\\
	&=\int_0^1 \left[ \, \int_{h^{-1}(\{ t\} ) \cap \Omega} |v'(h(y))|^p |\nabla h (y)|^{p-1} \,d \mathcal{H}^{N-1}(y) \,\right] \,dt.
\end{align*}
Using $|\nabla h| = \w_{N-1}^{\frac{1}{p-1}} h(y)^{\frac{N-1}{p-1}} |\nabla G (y)|$, 
\tcr{we have} 
\begin{align*}
	\int_{\Omega} |\nabla u |^p 
	=\int_0^1 \w_{N-1} t^{N-1} |v'(t)|^p\left[ \, \int_{h^{-1}(\{ t\} ) \cap \Omega}  |\nabla G (y)|^{p-1} \,d \mathcal{H}^{N-1}(y) \,\right] \,dt.
\end{align*}
Note that $h^{-1}(\{ t\} ) \cap \Omega$ is \tcr{also} a level set of $G$. 
Since
\begin{align*}
	\int_{\{  \tcr{G} < t \}} |\nabla \tcr{G}(y) |^{p} \,dy = t, \quad 
	\int_{\{  \tcr{G} = t \}} |\nabla \tcr{G}(y) |^{p-1} \,d\mathcal{H}^{N-1}(y) = 1.
\end{align*}
for any $t \in [0, \infty )$ 
(Ref. \tcr{\cite{F book} Lemma 9.1, or} \cite{CRN} Proposition 4), 
we obtain
\begin{align*}
\int_{h^{-1}(\{ t\} ) \cap \Omega}  |\nabla G (y)|^{p-1} \,d \mathcal{H}^{N-1}(y) 
= 1 \quad ({}^{\forall}t \in (0,1)),
\end{align*}
which implies that 
\begin{align*}
	\int_{\Omega} |\nabla u |^p =\int_0^1 \w_{N-1} t^{N-1} |v'(t)|^p \,dt = \int_{B_1^N} |\nabla v|^p.
\end{align*}
\end{proof}

\tcr{Up to now,} various transformations are \tcr{found in literature} so far. 
These transformations can be understood as a \tcr{variant} of harmonic transplantation.  
Here, we classify these transformations into the \tcr{following} three types:\\

\noindent
{\it I. \,\,\,\,Domains of two \tcr{Green's} functions in (\ref{Omega trans}) are different \tcr{from} each other. \\
II. \,\,Operators of two \tcr{Green's} functions in (\ref{Omega trans}) are different \tcr{from} each other. \\
III. Dimensions of two \tcr{Green's} functions in (\ref{Omega trans}) are different \tcr{from} each other.}\\

\noindent
Original harmonic transplantation (\ref{Omega trans}) is type I. 
For reader's convenience, we unify these transformations in the form of (\ref{Omega trans}) and summarize their \tcr{properties} briefly. 
In the present paper, we use harmonic transplantation (\ref{Omega trans}) in I.-(ii) \tcr{below}.  
\\

\noindent
{\it I.-(i): Critical case: $1< p = N$}

\noindent
If $\Omega = B_R^N, a=O$, then \tcr{the} harmonic transplantation \tcr{$u = H_a(v)$ in} (\ref{Omega trans}) becomes
\begin{align*}
u(y)= v(z), \,\text{where}\,\, G_{B_1^N, O}(y)= \w_{N-1}^{-\frac{1}{N-1}} \log \frac{1}{|y|} = \w_{N-1}^{-\frac{1}{N-1}} \log \frac{R}{|z|} = G_{B_R^N, O}(z)
\end{align*}
which coincides with the scaling $z= S_R (y) = Ry$. 
On the other hand, if $\Omega = \re_+^N$, $\tcr{a=(0,1)}$,
then \tcr{the} harmonic transplantation \tcr{$u = H_a(v)$ in} (\ref{Omega trans}) coincides with \tcr{the function ${\bf B}^{\sharp} v$ by} 
the Cayley type transformation \tcr{${\bf B}$} in (\ref{Cayley type}), see \S \ref{S Cayley}. 
\tcr{Similar to the various rearrangement techniques, 
harmonic transplantation (\ref{Omega trans}) enables us to construct appropriate test functions for various minimization or maximization problems:
we refer the readers to the application of harmonic transplantation to the study of} the Trudinger-Moser maximization problem on general bounded domain $\Omega$ (Ref. \cite{F, CR, L, CRN}). \\

\noindent
{\it I.-(ii): Subcritical case: $1<p<N$} 

\noindent
If $\Omega = B_R^N$ (let $B_\infty^N = \re^N$ and $\frac{1}{\infty} = 0$), $a=O$, 
then \tcr{the} harmonic transplantation \tcr{$u = H_a(v)$ in} (\ref{Omega trans}) becomes
\begin{align*}
	u(y)= v(z), \,\text{where}\,\, G_{B_1^N, O}(y)
	&= \frac{p-1}{N-p} \w_{N-1}^{-\frac{1}{p-1}} \left[ \,  |y|^{-\frac{N-p}{p-1}} -1  \,  \right] \\
	&= \frac{p-1}{N-p} \w_{N-1}^{-\frac{1}{p-1}} \left[ \,  |z|^{-\frac{N-p}{p-1}} -R^{-\frac{N-p}{p-1}}  \,  \right] 
	= G_{B_R^N, O}(z)
\end{align*}
which does not coincide with \tcr{$S_R^{\sharp}(v)$, here $S_R$ is the dilation} $z= S_R (y) = Ry \,(R<\infty)$, unlike I.-(i). 

We can obtain improved Hardy-Sobolev inequalities on $B_R^N$ via (\ref{Omega trans}), 
which are equivalent to the Hardy-Sobolev inequalities on $\re^N$ (Ref. \cite{I}. See also \cite{S(NA),S(ArXiv)}). 
Not only \tcr{the improvement of the inequalities}, 
but also a limit of the improved Hardy-Sobolev inequalities on $B_R^N$ as $p \nearrow N$ can be \tcr{considered}, 
unlike the classical \tcr{cases}.
For the subcritical Rellich inequality, a part of this argument still holds, see \S 2 in \cite{S(Rellich)}. 
A limit of \tcr{the} Hardy-Sobolev and the Poincar\'e \tcr{inequalities} (in some sense) can be considered, 
\tcr{see \cite{I} and \cite{BP} for taking a limit $p \nearrow N$ or $N \nearrow \infty$ in the Sobolev inequality respectively,
\cite{SS} for $p \nearrow N$ in the Hardy inequality, and $|\Omega| \searrow 0$ in the Poincar\'e inequality.
Also see \cite{S(RIMS)} for a survey.}


In the present paper, we consider \tcr{the} harmonic transplantation \tcr{$u = H_a(v)$, $v \in \dot{W}^{1,p}_{0, rad}(B^N_1)$, in} (\ref{Omega trans}) 
\tcr{for} $\Omega = \re_+^N$, \tcr{$a = e_N = (0,1)$, and $p \in (1,N)$.} 
Namely, 
\begin{align}\label{half-space trans}
	u(x,y)&= v(\tx, \ty), \,\text{where} \notag \\
	G_{\re^N_+, (0,1)} (x, y) 
	&= \frac{p-1}{N-p} \w_{N-1}^{-\frac{1}{p-1}} \left[ \,  \( |x|^2 +(1-y)^2 \)^{-\frac{N-p}{2(p-1)}} - \psi_p (x,y) \,\right] \notag \\
	&= \frac{p-1}{N-p} \w_{N-1}^{-\frac{1}{p-1}} \left[ \,  \( |\tx|^2 + |\ty|^2 \)^{-\frac{N-p}{2(p-1)}} -1  \,  \right]= G_{B_1^N, O}(\tx, \ty)
\end{align}
\tcr{where $\psi_p$ is as in \eqref{G half}.}

\begin{remark}\label{Rem two scaling}
\tcr{
We point out that, in the case $2=p<N$,
there are at least two transformations \tcr{$u$ of $v \in \dot{W}^{1,2}_{0,rad}(B^N_1)$}, 
by which $\| \nabla u \|_{L^2(\Omega)} = \| \nabla v \|_{L^2(B_1^N)}$ holds. 
Indeed, when $\Omega = B_R^N$ for $R < \infty$, the harmonic transplantation $u = H_0(v)$ and the M\"obius transformation $u = S_R^{\#}(v)$ 
from $\dot{W}^{1,2}_{0,rad}(B_1^N)$ to $\dot{W}^{1,2}_0(B_R^N)$ preserve the $L^2$ norm of the gradient.
Also when $\Omega = \re^N_{+}$, the harmonic transplantation $u = H_{e_N}(v)$ and the M\"obius transformation $u = {\bf B}^{\#}(v)$ 
via the Cayley type transformation (\ref{Cayley type}) from $\dot{W}^{1,2}_{0,rad}(B_1^N)$ to $\dot{W}^{1,2}_0(\re^N_+)$ have the same property, 
see \S \ref{S Cayley}. 
}
\end{remark}


\noindent
{\it II.}: {\it From weighted problem to \tcr{unweighted} problem}

\noindent
Let $1<p<N$ and $\tilde{G}_{B_1^N, O}$ be Green's function with singularity at $O$ associated with the weighted $p$-Laplace operator 
${\rm div }(|y|^{p-N} |\nabla (\cdot)|^{p-2} \nabla (\cdot))$. 
Define \tcr{$u: B^N_1 \to \re$ by}
\begin{align*}
	u(y)= v(z), \,&\text{where}\,\, \tilde{G}_{B_1^N, O}(y)
	=  \w_{N-1}^{-\frac{1}{p-1}}  \log \frac{1}{|y|} 
	= \frac{p-1}{N-p} \w_{N-1}^{-\frac{1}{p-1}} |z|^{-\frac{N-p}{p-1}}  
	= G_{\re^N, O}(z), \\
	&\tcr{y \in B_1^N, \, z \in \re^N.}
\end{align*}
Then we have $\| \nabla u \|_{L^p(B_1^N ;\, |y|^{p-N} \,dy)} = \| \nabla v \|_{L^p (\re^N)}$ (Ref. \cite{Z, HK}). 
We can remove the weight $|y|^{p-N}$ thanks to the above transformation. \\


\noindent
{\it III.-(i): From higher dimensions to one dimension}

\noindent
Consider \tcr{the} Moser transformation
\begin{align*}
u(y)= v(z), \,\text{where}\,\, G_{\tcr{B_R^N,} O}(y)
= \w_{N-1}^{-\frac{1}{N-1}} \log \frac{R}{|y|}
= z.\,
\quad \tcr{y \in B_R^N, z \in \re_{+}}.
\end{align*}
Then we have $\| \nabla u \|_{L^N(B_R^N)} = \| v' \|_{L^N (\re_+)}$. 
\tcr{The} Moser transformation is used to reduce the Trudinger-Moser maximization problem on $\dot{W}_{0, \text{rad}}^{1,N}(B_R^N)$ 
to the one-dimensional problem (Ref. \cite{M}).  
On the other hand, if we consider the Moser transformation on the subcritical Sobolev spaces $\dot{W}_{0, {\rm rad}}^{1,p}(B_R^N)\,(p< N)$, 
then we have
\begin{align*}
	u(y)= v(z), \,\text{where}\,\, G_{\tcr{B_R^N}, O}(y)
	= \frac{p-1}{N-p} \w_{N-1}^{-\frac{1}{p-1}} \left[ \, |y|^{-\frac{N-p}{p-1}} -R^{-\frac{N-p}{p-1}} \,\right]
	= z.\,
\end{align*}
Then \tcr{again} we have $\| \nabla u \|_{L^p(B_R^N)} = \| v' \|_{L^p (\re_+)}$.  
For an application of these transformations, see Proposition \ref{Prop Bliss} in \S \ref{S App}. 
\\


\noindent
{\it III.-(ii): Relation between the critical and the subcritical Sobolev spaces}

\noindent
Let $p=N < m$. If we consider \tcr{the relation} 
\begin{align}\label{trans dim}
	u(y)= v(z), \,\text{where}\,\, G_{B_R^N, O}(y) &= \w_{N-1}^{-\frac{1}{N-1}} \log \frac{R}{|y|} \notag \\
	&= \frac{N-1}{m-N} \w_{m-1}^{-\frac{1}{N-1}} \left[ \, |z|^{-\frac{m-N}{N-1}} -R^{-\frac{m-N}{N-1}} \,\right] = G_{B_R^m, O}(z),
\end{align}
\tcr{for $u \in \dot{W}^{1,N}_{0, rad}(B^N_R)$, $v \in \dot{W}^{1,N}_{0, rad}(B^m_R)$.} 
then we have $\| \nabla u \|_{L^N (B_R^N)} = \| \nabla v \|_{L^N (B_R^m)}$. 
Namely, we obtain the equality between two norms of the critical Sobolev spaces $\dot{W}_{0, \text{rad}}^{1,N} (B_R^N)$ 
and the higher dimensional subcritical Sobolev spaces $\dot{W}_{0, \text{rad}}^{1,p} (B_R^m)$ (Ref. \cite{ST}). 
This transformation (\ref{trans dim}) gives a direct relation between the subcritical Sobolev embeddings  
\begin{align*}
	\dot{W}_{0, \text{rad}}^{1,p} \hookrightarrow L^{p^*, p} \hookrightarrow L^{p^*, q} \hookrightarrow L^{p^*, \infty}
\end{align*}
\tcr{where $p < q$,}
and the critical Sobolev embeddings
\begin{align*}
	\dot{W}_{0, \text{rad}}^{1,N} \hookrightarrow L^{\infty, N}(\log L)^{-1} \hookrightarrow L^{\infty, q}(\log L)^{-1+ \frac{1}{N} - \frac{1}{q}} \hookrightarrow L^{\infty, \infty}(\log L)^{-1+\frac{1}{N}} = {\rm Exp L}^{\frac{N}{N-1}}.
\end{align*}
For the subcritical and the critical Sobolev embeddings, see e.g. \cite{SS} \S 1. 
\\

\noindent
{\it III.-(iii): An infinite dimensional form of the Sobolev inequality}

\noindent
Let $p< N < m$. If we consider 
\begin{align*}
	u(y)= v(z), \,&\text{where}\,\, G_{\re^N, O}(y)
	= \frac{p-1}{N-p} \w_{N-1}^{-\frac{1}{p-1}} |y|^{-\frac{N-p}{p-1}}
	= \frac{p-1}{m-p} \w_{m-1}^{-\frac{1}{p-1}} |z|^{-\frac{m-p}{p-1}}
	= G_{\re^m, O}(z), \\
	&\tcr{y \in \re^N, \, z \in \re^m}
\end{align*}
then we have $\| \nabla u \|_{L^p (\re^N)} = \| \nabla v \|_{L^p (\re^m)}$. 
Namely, we can reduce the $m$-dimensional Sobolev inequality: 
$S_{m,p} \| v \|^p_{L^{p^*}(\re^m)} \le \| \nabla v\|_{L^p(\re^m)}^p$ 
to \tcr{an} $N$-dimensional inequality \tcr{for $u$, which involves $m$ as a parameter and $m$ can be arbitrarily large}. 
Therefore, we can take a limit of the $m$-dimensional Sobolev inequality as $m \nearrow \infty$ \tcr{in this sense.} 
As a consequence, we can obtain the $N$-dimensional Hardy inequality: 
$\( \frac{N-p}{p} \)^p \int_{\re^N} \frac{|u|^p}{|y|^p} \, dy \le \int_{\re^N} |\nabla u|^p \,dy$ 
as an infinite dimensional form of the Sobolev inequality (Ref. \cite{S(NA)}).

%
%

\section{Proof of Theorems}\label{S Proof}

First, we show Theorem \ref{Thm CH}. 

\begin{proof}(Proof of Theorem \ref{Thm CH})
Let $p=N$. 
\tcr{We will ``transplant" the critical Hardy inequality on $B_1^N$ to $\re^N_{+}$ by the Cayley type transformation 
${\bf B}$ \eqref{Cayley type} in \S \ref{S Cayley}.}

\tcr{By \eqref{B maps to ball}, we consider ${\bf B}$ maps $\re^N_{+}$ to $B^N_1$.}
\tcr{Let $u \in C_c^1(\re^N_{+})$ and put $v(\tx, \ty) = ({\bf B}^{-1})^{\#}(u)(\tx, \ty)$ for $(\tx, \ty) \in B^N_1$. 
Then we see $u(x,y) = {\bf B}^{\#} (v)(x,y) = v({\bf B}(x,y))$ for $(x,y) \in \re^N_{+}$.} 
\tcr{From the fact that ${\bf B} \in M(\re^N)$} and Proposition \ref{Prop Mobius}, 
we have
\begin{align*}
	\int_{\re^N_+} |\nabla u\tcr{(x,y)}|^N \,dx d y = \int_{B_1^N} |\nabla v\tcr{(\tx, \ty)}|^N \,d\,\tx \, d\, \ty
\end{align*}
\tcr{where $(\tx, \ty) = {\bf B}(x, y)$.}
Since $|\tx|^2 + \ty^2 = \frac{|x|^2 + (y-1)^2}{|x|^2 + (y +1)^2}$ \tcr{by \eqref{B maps to ball}}, 
we have
\begin{align*}
	&\int_{B_1^N} \frac{|\tcr{v(\tx, \ty)}|^N}{\{ | \tx|^2 + |\ty|^2 \}^{\frac{N}{2}} 
	\( \log \frac{1}{\sqrt{ |\tx|^2 + |\ty|^2}} \)^N }\,d \,\tx \,d \,\ty \\
	&=\tcr{\int_{\re^N_+} \frac{|u(x,y)|^N}{|{\bf B}(x,y)|^{\frac{N}{2}} \( \log \frac{1}{|{\bf B}(x,y)|} \)^N} | \det {\bf B'}(x,y) | 
\,\,dx \,dy} \\
	&=\int_{\re^N_+}  \frac{|\tcr{u(x,y)}|^N  }{\{ |x|^2 + (1-y)^2 \}^{\frac{N}{2}} \left\{ \frac{|x|^2 + (1+y)^2}{4} \right\}^{\frac{N}{2}} \( \log \sqrt{ \frac{|x|^2+ (1+y)^2}{|x|^2 + (1-y)^2}} \)^N } \,dx dy\tcr{.}
\end{align*}
\tcr{Thus, we obtain \eqref{H_N half} by the critical Hardy inequality on the unit ball for $v$} (Ref. \cite{AS,II, TF}):
\begin{align*}
	\( \frac{N-1}{N} \)^N \int_{B_1^N} \frac{|\tcr{v(\tx, \ty)}|^N}{\{ |\tx|^2 + |\ty|^2 \}^{\frac{N}{2}} 
	\( \log \frac{1}{\sqrt{ |\tx|^2 + |\ty|^2}} \)^N }\,d \,\tx \,d \,\ty \le \int_{B_1^N} |\nabla \tcr{v(\tx, \ty)}|^N \,d\,\tx \, d\, \ty.
\end{align*}
\tcr{Optimality} and the non-attainability of the constant $\( \frac{N-1}{N} \)^N$ \tcr{in \eqref{H_N half}} follows 
from results for the critical Hardy inequality on the unit ball also.
\end{proof}

In the same way as above, we also obtain \tcr{a} Trudinger-Moser type inequality on the half-space from the result on balls 

\begin{theorem}\label{Thm TM}
Let $N \ge 2$. Then
\begin{align*}
	\sup \left\{ \int_{\re^N_+} e^{\alpha \,|u(x,y)|^{\frac{N}{N-1}}}  \,\frac{2^N \,dxdy}{\left\{  |x|^2 + (y+1)^2 \right\}^N} \,\,\middle| 
	\,\, \|\nabla u\|_{\tcr{L^N}(\re^N_+)} \le 1, \, u \in \dot{W}_{0}^{1,N}(\re^N_+) \,\, \right\} 
\end{align*}
is finite if and only if $\alpha \le N \w_{N-1}^{\frac{1}{N-1}}$. 
Moreover, the above maximization problem is attained for any $\alpha \le N \w_{N-1}^{\frac{1}{N-1}}$. 
\end{theorem}
\tcr{
\begin{proof}
Again, we use the transformation $v = ({\bf B}^{-1})^{\#}(u)$, $u \in \dot{W}^{1,N}_0(\re^N_{+})$.
Since $|{\rm det} \, {\bf B}'(x,y)| = \frac{2^N}{\left\{|x|^2 + (y+1)^2 \right\}^N}$, 
the theorem follows from the Trudinger-Moser inequality on $B^N_1$ and its attainability: see \cite{CC}. 
\end{proof}
}


Next, we show Theorem \ref{Thm IH}. 

\tcr{Before that, we claim the following Theorem:}

\begin{theorem}\label{Thm psi}
Let $1<p<N$. Then the Hardy type inequality
\begin{align*}
	\( \frac{N-p}{p} \)^p \int_{\re^N_+} \frac{W_p(x,y)^{\frac{p}{2}}}{\( |x|^2 +(1-y)^2 \)^{\frac{p}{2}}} |u(x,y)|^{p} \,dxdy 
	\le \int_{\re^N_+} |\nabla u(x,y)|^p \,dxdy,
\end{align*}
holds for any $u \in C_{c}^{1}(\re^N_+)$ \tcr{of the form $u(x,y) = u\(G_{\re^N_+, \tcr{e_N}}(x,y)\)$}, 
where $\psi_p$ is given in \S \ref{S Green} and for $(x,y) \in \re^N_+$,
\begin{align*}
	W_p(x,y) &= \frac{1+ \( |x|^2 +(1-y)^2 \)^{\frac{N-1}{p-1}} |\nabla \psi_p|^2 -2 \( |x|^2 +(1-y)^2 \)^{\frac{N-p}{2(p-1)}+1} \nabla \psi_p \cdot
	\begin{pmatrix}
	x \\
	y-1
	\end{pmatrix}
	}
	{\left[ 1-\tcr{\tilde{X}}^{\frac{N-p}{2(p-1)}} \right]^2}, \\
	\tcr{\tilde{X}} &= \psi_p (x,y)^{\frac{2(p-1)}{N-p}} \left\{ |x|^2 +(1-y)^2\right\}.
\end{align*}
\end{theorem}

\tcr{The proof of Theorem \ref{Thm psi} consists of the use of the harmonic transplantation 
between $\re^N_{+}$ and $\re^N$, the Hardy inequality (\ref{H_p}), and Lemma \ref{Lemma Omega}. 
Since the proof of Theorem \ref{Thm psi} is almost the same as that of Theorem \ref{Thm IH} below,} 
we omit the proof. 

As we mentioned in \S \ref{S Green},
we do not know the explicit form of $\psi_p$ when $2 \not= p < N$.
Due to the lack of the explicit form of $\psi_p$,
we cannot check that the inequality in Theorem \ref{Thm psi} is improved, i.e., $W_p(x,y) \ge 1$.  
Therefore, we consider a modification of harmonic transplantation by using $U_p$ \tcr{in \eqref{def U}} 
instead of $G_{\re^N_+, \tcr{e_N}}$ in (\ref{half-space trans}).
\tcr{This is the main idea we have invented in the proof of Theorem \ref{Thm IH}.}

Consider the following modified transformation for radial functions 
$v = v(z) = v(t)$, $(t= |z| \tcr{\in [0,1]})$ on $B_1^N$, or $w = w(\tilde{z})= w(r)$, $(r= |\tilde{z}| \tcr{\in [0, +\infty)})$ on $\re^N$:
\begin{align}\label{U trans}
\tcr{
	\begin{cases}
	&u(x,y) = v(t), \,\text{where}\,\,U_p (x,y) = G_{B_1^N, O}(t), \quad (x, y) \in \re^N_{+}, \\
	&u(x,y) = w(r), \,\text{where}\,\,U_p (x,y) = G_{\re^N, O}(r), \quad (x, y) \in \re^N_{+}.
	\end{cases}
}
\end{align}
\tcr{We call the function $u$ on $\re^N_{+}$ in \eqref{U trans} {\it the generalized harmonic transplantation} of $v$ (or $w$).} 

We obtain the following Lemma instead of Lemma \ref{Lemma Omega}. 

\begin{lemma}\label{Lemma Omega gene}
Let $1<  p < N, v, w$ be radial functions on $B_1^N, \re^N$, and $u$ be given by (\ref{U trans}). Then we have 
\begin{align*}
	\int_{\re^N_+} |\nabla u |^p 
	&= \int_{B_1^N} |\nabla v|^p + \int_0^1 \w_{N-1} t^{N-1} |v'(t)|^p  F_p \( G_{B_1^N, O} (t)  \)  \,dt\\
	&=\int_{\re^N} |\nabla w|^p + \int_0^\infty \w_{N-1} r^{N-1} |w'(r)|^p  F_p \( G_{\re^N, O}(r) \)  \,dr
\end{align*}
where 
\begin{equation}
\label{F_p}
	F_p(s)= \int_{[U_p > s] } (-\lap_p U_p) \, dxdy.
\end{equation}
Especially, 
\tcr{
if $p \ge 2$ (resp. $p \le 2$), then $F_p (s) \ge 0$ (resp. $F_p(s) \le 0$) 
and $\| \nabla u \|_p \ge \| \nabla v \|_p$, $\| \nabla u \|_p \ge \| \nabla w \|_p$ 
(resp. $\| \nabla u \|_p \le \| \nabla v \|_p$, $\| \nabla u \|_p \le \| \nabla w \|_p$) 
holds.}
\end{lemma}

\begin{proof}
\tcr{We prove the first equality only, since the proof of the second equality is similar.}
\tcr{Also we note that the proof below is an analogue to that of Lemma \ref{Lemma Omega}.} 
Let $h$ be defined by 
$$
	\tcr{h(x,y) = \(G_{B^N_1, O}\)^{-1} \( U_p(x,y) \)} = \left[ \, \frac{N-p}{p-1} \w_{N-1}^{\frac{1}{p-1}} U_p (x,y) +1  \, \right]^{-\frac{p-1}{N-p}}, 
	\quad \tcr{(x, y) \in \re^N_{+}}.
$$
\tcr{Thus $t = h(x,y)$ is equivalent to $U_p(x,y) = G_{B^N_1, O}(t)$}
and $u(x,y)= v(h(x,y))$.
In particular, $\nabla u(x, y)= v'(h(x,y)) \nabla h(x,y)$.
\tcr{Note that since $U_p \ge 0$ in $\re^N_+$, 
we get that $0 < h(x,y) \le 1$ for $(x, y) \in \ol{\re^N_{+}}$ and} if $(x,y) \in h^{-1}(\{ t\} ) \cap \re^N_{+}$, 
Thus, the coarea formula gives that 
\begin{align*}
	\int_{\re^N_+} |\nabla u |^p 
	&= \int_{\re^N_+} |v'(h(x,y))|^p |\nabla h (x, y)|^{p-1} |\nabla h (x,y)| \,dxdy\\
	&=\int_0^1 \left[ \, \int_{h^{-1}(\{ t\} ) \cap \re^N_+} |v'(h(x,y))|^p |\nabla h (x,y)|^{p-1} \,d \mathcal{H}^{N-1}(x,y) \,\right] \,dt.
\end{align*}
\tcr{Inserting} $|\nabla h| = \w_{N-1}^{\frac{1}{p-1}} h(y)^{\frac{N-1}{p-1}} |\nabla U_p (x,y)|$, 
\tcr{we have}
\begin{align}
\label{LL1}
	\int_{\re^N_+} |\nabla u |^p 
	=\int_0^1 \w_{N-1} t^{N-1} |v'(t)|^p
	\left[ \, \int_{h^{-1}(\{ t\} ) \cap \re^N_+} |\nabla U_p (x, y)|^{p-1} \,d \mathcal{H}^{N-1}(x,y) \,\right] \,dt.
\end{align}
Note that $h^{-1}(\{ t\} ) \cap \re^N_+$ is a level set of $U_p$. 
Applying $\phi_t = {\rm min} \{ t, U_p \}$, $(t >0)$ as a test function \tcr{of \eqref{eq:sol U}} in Proposition \ref{Prop sol U},
\tcr{we have}
\begin{align*}
	\int_{[U_p < t]} | \nabla U_p |^{p} \, dxdy 
	= t + \int_{[\tcr{U_p} < t] } (-\lap_p U_p)\, U_p\, dxdy + \int_{[U_p > t] } (-\lap_p U_p)\,t \, dxdy.
\end{align*}
If we differentiate the above with respect to $t$, then we have
\begin{align*}
	&\int_{[U_p = t]} | \nabla U_p |^{p-1} \, d \mathcal{H}^{N-1}(x,y)  \\
	&= 1 + \int_{[U_p = t] } (-\lap_p U_p)\, \frac{U_p}{|\nabla U_p|}\, d \mathcal{H}^{N-1}(x,y) 
	- \int_{[U_p = t] } (-\lap_p U_p)\,\frac{t}{|\nabla U_p|} \, d \mathcal{H}^{N-1}(x,y) \\
	&\hspace{2em} + \int_{[U_p > t] } (-\lap_p U_p) \, dxdy \\
	&= 1 +  \int_{[U_p > t] } (-\lap_p U_p) \, dxdy
\end{align*}
thanks to the coarea formula.  
Therefore, \tcr{replacing $t$ by $G_{B^N_1, O}(t)$} for any $t \in (0,1)$, 
we have
\begin{align}
\label{LL2}
	\int_{h^{-1}(\{ t\} ) \cap \re^N_+}  |\nabla U_p (x,y)|^{p-1} \,d \mathcal{H}^{N-1}(x,y) 
	= 1 +  \int_{ \left[ U_p > G_{B^N_1, O}(t) \right] } (-\lap_p U_p) \, dxdy. 
\end{align}
\tcr{Inserting \eqref{LL2} into \eqref{LL1}, we obtain}
\begin{align*}
	&\int_{\re^N_+} |\nabla u |^p 
	= \int_0^1 \w_{N-1} t^{N-1} |v'(t)|^p \( 1+ \int_{[U_p > t] } (-\lap_p U_p) \, dxdy\) \,dt \\
	&= \int_{B_1^N} |\nabla v|^p + \int_0^1 \w_{N-1} t^{N-1} |v'(t)|^p  
	F_p \( \w_{N-1}^{-\frac{1}{p-1}} \frac{p-1}{N-p} \( t^{-\frac{N-p}{p-1}} -1  \) \) \,dt.
\end{align*}
\end{proof}


\begin{lemma}\label{Lemma IH gene}
Let $1<  p < N$, $0 \le s \le p$, \tcr{$v$ and $w$} be radial functions on $B_1^N$ and $\re^N$ \tcr{respectively}, 
and \tcr{let} $u$ be given by (\ref{U trans}). 
Then we have 
\begin{align*}
	&\int_{\re^N_+} \frac{V_p(x,y)^{\frac{p}{2}} }{\( |x|^2 +(1-y)^2 \)^{\frac{p}{2}}} |u(x,y)|^{p} \,dxdy \\
	&=\int_{B_1^N} \frac{|v|^{p}}{|z|^p \left[ 1-|z|^{\frac{N-p}{p-1}} \right]^p} \,dz
	+ \w_{N-1} \int_0^1  |v|^{p} t^{N-1-p} \left[ 1-t^{\frac{N-p}{p-1}} \right]^{-p} F_p \( G_{B_1^N, O}(t) \) \,dr \\
	&= \int_{\re^N} \frac{|w|^{p}}{|\tilde{z}|^p}\,d\tilde{z} 
	+ \w_{N-1} \int_0^\infty  |w|^{p} r^{N-1-p} F_p \( G_{\re^N, O}(r) \) \,dr.
\end{align*}
where \tcr{$F_p(s)$ is defined in \eqref{F_p},}
and for $(x,y) \in \re^N_+$, \tcbb{$V_p$ and $X$ are defined in \eqref{V_p}.}
\end{lemma}

\begin{proof}
\tcr{We prove the second equality for $v$ only. The first equality is similar.}
\tcr{For $(x, y) \in \re^N_{+}$, define $h(x,y)$ by the relation}
\begin{align*}
	\tcr{U_p (x, y) = G_{\re^N, O}(h(x,y))} = \frac{p-1}{N-p} \w_{N-1}^{-\frac{1}{p-1}} h(x,y)^{-\frac{N-p}{p-1}},
\end{align*}
\tcr{that is,}
\begin{align*}
	h(x,y) = \left[ \,\( |x|^2 +(1-y)^2 \)^{-\frac{N-p}{2(p-1)}} - \( |x|^2 +(1+y)^2 \)^{-\frac{N-p}{2(p-1)}} \,\right]^{\frac{p-1}{p-N}}.
\end{align*}
\tcr{As in \eqref{LL2}, we can obtain}
\begin{align}
\label{LL3}
	\tcr{\int_{h^{-1}(\{ t\} ) \cap \re^N_+}  |\nabla U_p (x,y)|^{p-1} \,d \mathcal{H}^{N-1}(x,y) 
	= 1 +  \int_{ \left[ U_p > G_{\re^N, O}(t) \right] } (-\lap_p U_p) \, dxdy.} 
\end{align}
\tcr{Thus by the coarea formula and (\ref{LL3}),} we have
\begin{align}
\label{LL4}
	&\int_{\re^N_+} |u|^{p}  \frac{|\nabla h(x,y)|^p}{h(x,y)^p} \,dx dy
	= \int_0^\infty \int_{h^{-1}(\{ r\} ) \cap \re^N_+} |w|^{p} r^{-p} |\nabla h (x,y)|^{p-1} \,d \mathcal{H}^{N-1}(x,y) \,dr  \\
	&= \,\w_{N-1} \int_0^\infty  |w|^{p} r^{N-1-p} \int_{h^{-1}(\{ r \} ) \cap \re^N_+} 
	|\nabla U_p (\tcr{x,y})|^{p-1} \,d \mathcal{H}^{N-1}(x,y) \,dr \notag \\
	&\overset{\eqref{LL3}}{=} \w_{N-1} \int_0^\infty |w|^{p} r^{N-1-p} \,dr \tcr{+} \w_{N-1} \int_0^\infty  |w|^{p} r^{N-1-p} 
	F_p \( \tcr{G_{\re^N, O} (r)} \) \,dr \notag \\
	&=\int_{\re^N} \frac{|w|^{p}}{|\tilde{z}|^p}\,d\tilde{z}
	+ \w_{N-1} \int_0^\infty  |w|^{p} r^{N-1-p} F_p \( G_{\re^N, O} (r) \)   \,dr, \notag
\end{align}
where $|\tilde{z}| = r=h(x,y)$.
On the other hand, we have
\begin{align*}
	\nabla h(x,y) =  h(x,y)^{\frac{N-1}{p-1}} \left[ \,  \( |x|^2 +(1-y)^2 \)^{\frac{p-N}{2(p-1)}-1}
	\begin{pmatrix}
	x \quad \\
	y-1
	\end{pmatrix}
	- \( |x|^2 +(1+y)^2 \)^{\frac{p-N}{2(p-1)}-1}
	\begin{pmatrix}
	x \quad \\
	y+1
	\end{pmatrix}
	\,\right]
\end{align*}
which implies that
\begin{align*}
	&\frac{|\nabla h(x,y) |^p}{h(x,y)^p}
	=  h(x,y)^{\frac{N-p}{p-1}p} 
	\left| \,  \( |x|^2 +(1-y)^2 \)^{\frac{p-N}{2(p-1)}-1}
	\begin{pmatrix}
	x \quad \\
	y-1
	\end{pmatrix}
	- \( |x|^2 +(1+y)^2 \)^{\frac{p-N}{2(p-1)}-1}
	\begin{pmatrix}
	x \quad \\
	y+1
	\end{pmatrix}
	\,\right|^{p}\\
	&=\frac{
	\left[ \,  \( |x|^2 +(1-y)^2 \)^{\frac{1-N}{p-1}} 
	+ \( |x|^2 +(1+y)^2 \)^{\frac{1-N}{p-1}} -2 \( |x|^2 +(1-y)^2 \)^{\frac{2-p-N}{2(p-1)}} 
	\( |x|^2 +(1+y)^2 \)^{\frac{2-p-N}{2(p-1)}} (|x|^2 +y^2 -1)  \,\right]^{\frac{p}{2}}
	}{
	\left[ \,  \( |x|^2 +(1-y)^2 \)^{-\frac{N-p}{2(p-1)}} - \( |x|^2 +(1+y)^2 \)^{-\frac{N-p}{2(p-1)}} \,\right]^{p} }\\
	&= \frac{\left[ 1+ X^{\frac{N-1}{p-1}} -2X^{\frac{N-p}{2(p-1)}} \( |x|^2 +(1+y)^2 \)^{\tcrr{-1}} 
	(|x|^2 +y^2 -1) \right]^{\frac{p}{2}}}{ \left[ |x|^2 +(1-y)^2 \right]^{\frac{p}{2}} \left[ 1-X^{\frac{N-p}{2(p-1)}} \right]^p} \\
	&= \frac{V_p(x,y)^{\frac{p}{2}}}{ \left[ |x|^2 +(1-y)^2 \right]^{\frac{p}{2}} }.
\end{align*}
\tcr{Inserting this in the left hand-side of \eqref{LL4}, we obtain the second equality of} Lemma \ref{Lemma IH gene}. 
\end{proof}


\tcr{Now, we prove Theorem \ref{Thm IH}.}
\tcr{First, we claim the next lemma.}
\begin{lemma}
\label{Lemma:claim 2}
Let $2 \le p \le N$ and let $F_p$ be defined in \eqref{F_p}.
Then
\begin{align}\label{claim 2}
	\( \frac{N-p}{p} \)^p \int_0^\infty  |w|^{p} r^{N-1-p} F_p \( G_{\re^N, O} (r)  \) \,dr 
	\le \int_0^\infty  |w'|^{p} r^{N-1} F_p \( G_{\re^N, O} (r)  \)  \,dr
\end{align}
holds for any radial function $w = w(r) \in C_c^1(\re^N)$.
\end{lemma}

\begin{proof}
Since for $p \ge 2$,
\begin{align*}
	F'_p (s) = - \int_{[U_p=s]} \frac{-\lap_p U_p}{ |\nabla U_p|} d\mathcal{H}^{N-1} (x,y) \le 0
\end{align*}
\tcr{by Proposition \ref{Prop cal U}} and $G^{'}_{\re^N, O} (r) \le 0$, 
we have 
\begin{align*}
	&\int_0^\infty  |w|^{p} r^{N-1-p} F_p \( G_{\re^N, O} (r) \)   \,dr \\
	&= -\frac{p}{N-p} \int_0^\infty  |w|^{p-2} w w' r^{N-p} F_p \( G_{\re^N, O} (r) \) \,dr 
	- \frac{1}{N-p} \int_0^\infty  |w|^{p} r^{N-p} F^{'}_p \( G_{\re^N, O} (r) \) G^{'}_{\re^N, O} (r) \,dr \\
	&\le -\frac{p}{N-p} \int_0^\infty  |w|^{p-2} w w' r^{N-p} F_p \( G_{\re^N, O} (r) \)   \,dr \\
	&\le \frac{p}{N-p} \( \int_0^\infty  |w'|^{p} r^{N-1} F_p \( G_{\re^N, O} (r) \) \,dr \)^{\frac{1}{p}} 
	\( \int_0^\infty  |w|^p r^{N-1-p} F_p \( G_{\re^N, O} (r) \)   \,dr \)^{\frac{p-1}{p}}.
\end{align*}
Therefore, we obtain (\ref{claim 2}).
\end{proof}

\begin{proof}(Proof of Theorem \ref{Thm IH})
From Lemma \ref{Lemma IH gene}, \eqref{claim 2}, the classical Hardy inequality (\ref{H_p}) on $\re^N$, and Lemma \ref{Lemma Omega gene},
we have
\begin{align*}
	&\( \frac{N-p}{p} \)^p \int_{\re^N_+} \frac{|u(x,y)|^{p}}{\( |x|^2 +(1-y)^2 \)^{\frac{p}{2}}} V_p(x,y)^{\frac{p}{2}} \,dxdy \\
	&\overset{{\rm Lemma} \, \ref{Lemma IH gene}}{=} \( \frac{N-p}{p} \)^p\int_{\re^N} \frac{|w|^{p}}{|\tilde{z}|^p}\,d\tilde{z} 
	+ \( \frac{N-p}{p} \)^p\w_{N-1} \int_0^\infty  |w|^{p} r^{N-1-p} F_p \( G_{\re^N, O}(r) \) \,dr \\
	&\overset{\eqref{claim 2}}{\le} \( \frac{N-p}{p} \)^p\int_{\re^N} \frac{|w|^{p}}{|\tilde{z}|^p}\,d\tilde{z} 
	+ \int_0^\infty  |w'|^{p} r^{N-1} F_p \( G_{\re^N, O} (r) \)   \,dr \\
	&\overset{\eqref{H_p}}{<} \int_{\re^N} |\nabla w|^p\,d\tilde{z} 
	+ \int_0^\infty  |w'|^{p} r^{N-1} F_p \( G_{\re^N, O} (r)  \)  \,dr \\ 
	&\overset{{\rm Lemma} \, \ref{Lemma Omega gene}}{=} \int_{\re^N_+} |\nabla u (x,y) |^p \, dxdy.
\end{align*}
\tcr{Thus the inequality \eqref{IH}
\begin{align*}
	&\( \frac{N-p}{p} \)^p\int_{\re^N_+} \frac{|u(x,y)|^p}{\( |x|^2 +(1-y)^2 \)^{\frac{p}{2}}} V_p(x,y)^{\frac{p}{2}} \,dxdy 
	< \int_{\re^N_+} |\nabla u (x,y)|^p \, dxdy.
\end{align*}
is proven.}

The remaining is to show the optimality of the constant $\( \frac{N-p}{p} \)^p$ in the above inequality. 
For large $M >0$ and small $\ep >0$, consider the following test function:
\begin{align*}
	u_{\ep, M} (x,y)= U_p (x,y)^{\frac{p-1}{p} - \frac{p-1}{N-p} \ep} \psi_M \( U_p (x,y) \)
\end{align*}
where $U_p (x,y)$ \tcr{is in \eqref{def U}.}
\tcr{Put} 
$$
	C(N,p) = \frac{p-1}{N-p} \w_{N-1}^{-\frac{1}{p-1}}
$$ 
and define $\psi_M \in C^\infty (0,\infty)$, $0\le \psi_M \le 1$, $\psi_M (s) =1$ for $s \ge M$, $\psi_M (s) =0$ for $s \le \frac{M}{2}$. 
Let $\delta_1 = \delta_1 (M)$, $\delta_2 =\delta_2(M) >0$ satisfy 
\begin{align*}
	\delta_1^{-\frac{N-p}{p-1}} - (2-\delta_1)^{-\frac{N-p}{p-1}} = C(N,p)^{-1} M, \quad
	\delta_2 = \( \frac{M}{C(N,p)} \)^{\frac{p-1}{N-p}}. 
\end{align*}
Then we have $B_{\delta_1} = B_{\delta_1}(\tcr{0,1}) \subset [U_p \ge M] \subset B_{\delta_2}$. 
Then we have
\begin{align}\label{deno}
	&\int_{\re^N_+} \frac{|u_{\ep, M}(x,y)|^{p} V_p(x,y)^{\frac{p}{2}}}{\( |x|^2 +(1-y)^2 \)^{\frac{p}{2}}}  \,dxdy 
	\ge \int_{[U_p \ge M]}  \frac{|U_p|^{p-1-\frac{p-1}{N-p} p\ep} }{\( |x|^2 +(1-y)^2 \)^{\frac{p}{2}}}  \,dxdy \notag \\
	&\ge C(N,p)^{p-1-\frac{p-1}{N-p} p\ep} \int_{B_{\delta_1}} \( |x|^2 +(1-y)^2 \)^{-\frac{N}{2} + \frac{p\ep}{2} } 
	\left[ 1-\(  \frac{|x|^2 +(y-1)^2}{|x|^2 + (y+1)^2} \)^{\frac{N-p}{2(p-1)}}  \right]^{p-1} \, dxdy \notag \\
	&\ge C(N,p)^{p-1-\frac{p-1}{N-p} p\ep} \int_{B_{\delta_1}} \( |x|^2 +(1-y)^2 \)^{-\frac{N}{2} + \frac{p\ep}{2} } 
	\left[ 1- (p-1) \(  \frac{|x|^2 +(y-1)^2}{|x|^2 + (y+1)^2} \)^{\frac{N-p}{2(p-1)}}  \right] \, dxdy \notag \\
	&\ge C(N,p)^{p-1-\frac{p-1}{N-p} p\ep} \w_{N-1} \left[ \int_0^{\delta_1} r^{-1+ p\ep} \,dr - \frac{p-1}{(2 -\delta_1)^{\frac{N-p}{p-1}}} 
	\int_0^{\delta_1} r^{-1 + p\ep + \frac{N-p}{p-1}} \,dr \right] \notag \\
	&= \frac{C(N,p)^{p-1}}{p} \w_{N-1} \ep^{-1} + o(\ep^{-1}) \quad (\ep \to 0).
\end{align}
Since 
\begin{align*}
	\nabla u_{\ep, M} = \( \frac{p-1}{p}  -\frac{p-1}{N-p} \ep \) U_p^{-\frac{1}{p} -\frac{p-1}{N-p} \ep} \, (\nabla U_p ) \, \psi_M 
	+ \psi^{'}_M \, U_p^{1-\frac{1}{p} -\frac{p-1}{N-p} \ep} \, (\nabla U_p )
\end{align*}
\tcr{and $(a + b)^p \le a^p + p a^{p-1}b$ for $a, b \ge 0$},
we have
\begin{align*}
	&|\nabla u_{\ep, M} |^p \\
	&\le \( \frac{p-1}{p}  -\frac{p-1}{N-p} \ep \)^p U_p^{-1 -\frac{p-1}{N-p} p \ep} \, |\nabla U_p |^p  
	+ p \( \frac{p-1}{p} \)^{p-1} \psi^{'}_M \, U_p^{-\frac{p-1}{N-p} \ep  -\frac{(p-1)^2}{N-p} \ep} \, |\nabla U_p |^p.
\end{align*}
Then we have
\begin{align}\label{nume}
	&\int_{\re^N_+} |\nabla u_{\ep, M}(x,y)|^{p} \,dxdy \notag \\
	&\le \int_{[U_p \ge M]}  \( \frac{p-1}{p}  -\frac{p-1}{N-p} \ep \)^p U_p^{-1 -\frac{p-1}{N-p} p \ep} \, |\nabla U_p |^p \,dxdy 
	+ \int_{[\frac{M}{2} \le U_p \le M]} |\nabla u_{\ep, M}|^{p} \,dxdy \notag \\
	&\le \( \frac{p-1}{p} \)^p C(N,p)^{-1 -\frac{p-1}{N-p}p\ep} \w_{N-1}^{-\frac{p}{p-1}} 
	\int_{B_{\delta_2}} \( |x|^2 +(1-y)^2 \)^{-\frac{N}{2} + \frac{p\ep}{2} } \tcr{\times} \notag \\
	&\hspace{1em} \Biggl[1 + \left\{  \frac{|x|^2 + (y-1)^2}{(2-\delta_2)^2} \right\}^{\frac{N-p}{p-1} +1}
	+ 2 \left\{  \frac{|x|^2 + (y-1)^2}{(2-\delta_2)^2} \right\}^{\frac{N-p}{2(p-1)} } 
	\frac{ \( |x|^2 +y^2 -1 \)_-}{(2-\delta_2)^2} \Biggl] \,dxdy +o(\ep^{-1}) \notag \\
	&\le \( \frac{p-1}{p} \)^p C(N,p)^{-1 -\frac{p-1}{N-p}p\ep} \w_{N-1}^{-\frac{p}{p-1}+1} 
	\int_0^{\delta_2} r^{-1 + p\ep} \, dr  +o(\ep^{-1}) \notag \\
	&= \( \frac{p-1}{p} \)^p \frac{C(N,p)^{-1}}{p} \w_{N-1}^{-\frac{1}{p-1}} \ep^{-1} + o(\ep^{-1}) \quad (\ep \to 0),
\end{align}
where $(\,f(x) \,)_- := \max \{ 0, \, -f(x) \}$. 
From (\ref{deno}) and (\ref{nume}), we have
\begin{align*}
	\frac{\int_{\re^N_+} |\nabla u_{\ep, M}(x,y)|^{p} \,dxdy}{\int_{\re^N_+} \frac{|u_{\ep, M}(x,y)|^{p} V_p(x,y)^{\frac{p}{2}}}
	{\( |x|^2 +(1-y)^2 \)^{\frac{p}{2}}}  \,dxdy } 
	&\le \frac{\( \frac{p-1}{p} \)^p \frac{C(N,p)^{-1}}{p} \w_{N-1}^{-\frac{1}{p-1}} \ep^{-1} + o(\ep^{-1})}
	{\frac{C(N,p)^{p-1}}{p} \w_{N-1} \ep^{-1} + o(\ep^{-1})}  \\
	&= \( \frac{p-1}{p\,C(N,p)} \)^p \w_{N-1}^{-\frac{p}{p-1}}+ o(1) \\
	&= \( \frac{N-p}{p} \)^{p} + o(1) \quad (\ep \to 0).
\end{align*}
Therefore, the constant $\( \frac{N-p}{p} \)^p$ in the inequality (\ref{IH}) is optimal. 
\end{proof}


\tcrr{As we mention in Remark \ref{Rem without sym}, the improved inequality (\ref{IH}) is valid for functions without any symmetry by using Proposition \ref{Prop sol U} and a result in \cite{DD}. }

\begin{theorem}\label{Thm IH without sym}
Let $2 \le p < N$. 
Then the inequality
\begin{align}
\label{IH}
	\( \frac{N-p}{p} \)^p \int_{\re^N_+} \frac{V_p(x,y)^{\frac{p}{2}}}{\( |x|^2 +(1-y)^2 \)^{\frac{p}{2}}} |u(x,y)|^p \,dxdy 
	&\le  \int_{\re^N_+} \left| \nabla u(x,y)  \right|^p \,dxdy
\end{align}
holds for any $u \in \dot{W}_0^{1,p} (\re^N_+)$, 
where \tcbb{$V_p$ and $X$ are defined in \eqref{V_p}.} 
Furthermore, $(\frac{N-p}{p})^p$ is the best constant and is not attained.
\end{theorem}

\begin{proof}
\tcrr{
$U_p$ is a nonnegative function and Proposition \ref{Prop sol U} implies that $-\lap_p U_p \ge 0$ in weak sense for $p \in [2, N)$. Substituting $U_p$ for $\rho$ in \cite{DD}:Theorem 2.1, we have the inequality (\ref{IH}) for any functions $\dot{W}_0^{1,p} (\re^N_+)$, since
\begin{align*}
\( \frac{p-1}{p} \)^p \frac{|\nabla \rho |^p}{\rho^p} 
&=\( \frac{p-1}{p} \)^p \frac{|\nabla U_p |^p}{U_p^p} \\
&= \( \frac{p-1}{p} \)^p \( \frac{N-p}{p-1} \)^p \frac{|\nabla h(x,y) |^p}{h(x,y)^p} \\
&= \( \frac{N-p}{p} \)^p \frac{V_p(x,y)^{\frac{p}{2}}}{ \left[ |x|^2 +(1-y)^2 \right]^{\frac{p}{2}} },
\end{align*}
where $h(x,y)$ is given by the proof of Lemma \ref{Lemma IH gene}. 
The optimality of the constant $(\frac{N-p}{p})^p$ in the inequality (\ref{IH}) follows from Theorem \ref{Thm IH} and the} \tcbb{non-attainability} \tcrr{follows from \cite{DD}:Theorem 4.1.}
\end{proof}


In the last of this section, we give \tcr{an} improved Hardy-Sobolev \tcr{inequality} on the half-space for $p=2$. 
The proof is simpler than the \tcr{that} of Theorem \ref{Thm IH}. 
We omit \tcr{it} here.

\begin{theorem}\label{Thm IHS}(Improved Hardy-Sobolev \tcr{inequality} for $p =2$)
Let \tcr{$p=2 < N$}, $0 \le s <2$, and $2^*(s) = \frac{2 (N-s)}{N-2}$. 
Then the \tcr{inequality}
\begin{align}\label{IHS}
	S_{N,2,s} &\( \int_{\re^N_+} \frac{|u(x,y)|^{2^*(s)}}{\( |x|^2 +(1-y)^2 \)^{\frac{s}{2}}} 
	\frac{V_2(x,y)}{ \left[ 1-X^{\frac{N-2}{2}} \right]^{\frac{2-s}{N-2}}}  \,dxdy \)^{\frac{2}{2^* (s)}} 
	\le \int_{\re^N_+} |\nabla u(x,y)|^2 \,dxdy,
\end{align}
\tcr{holds} for any $u \in \dot{W}_0^{1,2}(\re^N_+)$ 
\tcr{of the form $u(x,y)= \tilde{u}(G_{\re^N_+, \tcr{e_N}}(x,y))$ for some function $\tilde{u}$ on $[0, +\infty)$,} 
where $V_2(x,y)$ and $X$ is given by \eqref{V_p} in Theorem \ref{Thm IH}, 
and $S_{N,2,s}$ is the Hardy-Sobolev best constant, i.e., 
\begin{align*}
	S_{N,2,s} 
	= \inf_{u \in C_c^\infty(\re^N) \setminus \{ 0\}} \frac{\int_{\re^N} |\nabla u|^2 \,dx}{\( \int_{\re^N} \frac{|u|^{2^*(s)}}{|x|^s} \,dx \)^{\frac{2}{2^*(s)}}}.
\end{align*}
\end{theorem}

\begin{remark}
\tcrr{The inequality (\ref{IHS}) does not hold \tcr{for functions} without the symmetry (\ref{*'}), 
see Proposition \ref{Prop all fct. zero}.}
\end{remark}

%
%

\section{Appendix}\label{S App}

First, we show that 
the radial critical Sobolev \tcr{space} $\dot{W}_{0, {\rm rad}}^{1,N} (\re^N)$ cannot be embedded 
to \tcr{any} weighted Lebesgue \tcr{space} $L^q (\re^N; g(x) \,dx)$ for $q \in [1, \infty)$ and for $g >0$. 

\begin{prop}\label{Prop re^N}
There is no weight function $g >0$ such that the inequality
\begin{align*}
C \( \int_{\re^N} |u|^q g(x) \,dx \)^{\frac{N}{q}} \le \int_{\re^N} |\nabla u|^N \,dx
\end{align*}
holds for any $u \in C_{c, {\rm rad}}^1(\re^N)$ for some $C>0$. 
\end{prop}

\begin{proof}
Consider the radial test function
\begin{align*}
	\phi_R(|x|) = 
	\begin{cases}
	1 \quad&\text{if} \,\, |x| \le 1,\\
	\frac{\log \frac{R}{|x|}}{\log R} &\text{if} \,\, 1< |x| < R,\\
	0 &\text{if} \,\, |x| \ge R.
	\end{cases}
\end{align*}
Direct calculation shows that
\begin{align*}
	&\int_{\re^N} |\phi_R |^q g(x) \,dx \ge \int_{B_1} g(x) \,dx >0, \\
	&\int_{\re^N} |\nabla \phi_R|^N \,dx = \w_{N-1} \( \log R \)^{1-N} \to 0 \quad (R \to \infty).
\end{align*}
\tcr{Though $\phi_R$ is not $C^1$, we can mollify it as in \cite{HK}:Lemma 8.1, to obtain a $C^1_{rad}(\re^N)$ function with the same property.}
Therefore, we obtain Proposition \ref{Prop re^N}. 
\end{proof}


\tcr{Next}, we show \tcr{that} the improved inequalities (\ref{IH}), (\ref{IHS}) in Theorem \ref{Thm IH} \tcr{and Theorem} \ref{Thm IHS} 
do not hold without the symmetry (\ref{*'}).

\begin{prop}\label{Prop all fct. zero}
Let $1<p<N$, $0 \le s \,\tcrr{<} \, p$, $p^*(s )= \frac{p(N-s)}{N-p}$, 
and $V_p$, \tcr{$X$ }be given in \tcr{\eqref{V_p}}. 
Then
\begin{align*}
	S:= \inf_{u \in C_c^1(\re^N_+) \setminus \{ 0\}} 
	\frac{\int_{\re^N_+} |\nabla u(x,y)|^p \,dxdy}{\( \int_{\re^N_+} \frac{|u(x,y)|^{p^*(s)}}
	{\( |x|^2 +(1-y)^2 \)^{\frac{s}{2}}} \frac{V_p(x,y)^{\frac{p}{2}}}{\left[ 1-X^{\frac{N-p}{2(p-1)}} \right]^{\frac{(p-1)(p-s)}{N-p}}} \,dxdy \)^{\frac{p}{p^* (s)}}} 
	= 0\tcr{.}
\end{align*}
\end{prop}

\tcrr{
\begin{proof}
We use the same test function as it in \cite{S(NA)}:Propositon 2. 
Let $z=(x,y) \in \re^N_+$ and $z_\ep = (0, \ep)$. 
Note that $X = \frac{|x|^2 + (1-y)^2}{|x|^2 + (1+y)^2} \to 1$ and 
\begin{align*}
V_p (x, y) &= \( \frac{4(p-1)}{N-p} \)^2 (1-X)^{-2} + o\( (1-X)^{-2} \) \\
&=\( \frac{p-1}{N-p} \)^2 y^{-2} + o\( y^{-2} \)
\end{align*}
as $|z| = \sqrt{ |x|^2 + y^2} \to 0$. 
For small $\ep >0$, we define $u_\ep$ as follows:
\begin{align*}
u_{\ep}(z) =
	   \begin{cases}
		v\( \frac{|z-z_{\ep}|}{\ep} \) \,\,\,&\text{if} \,\,\, z \in B_{\ep}(z_{\ep}), \\
		0  &\text{if} \,\,\, z \in \re^N_+ \setminus B_{\ep}(z_{\ep}),
	   \end{cases}
\,\, \text{where}\,\, v(t)=
	   \begin{cases}
		1 \,\,\,&\text{if} \,\,\,0\le t \le \frac{1}{2},  \\
		2(1-t)  &\text{if} \,\,\, \frac{1}{2} < t \le 1.
	   \end{cases}
\end{align*}
Then we have 
\begin{align*}
&\int_{\re^N_+} | \nabla u_{\ep} (z)|^p \,dz = \ep^{N-p} \int_{B_1} |\nabla v (|z|)|^p \,dz = C \ep^{N-p}, \\
	&\int_{\re^N_+} \frac{|u_\ep (x,y)|^{p^*(s)}}{\( |x|^2 +(1-y)^2 \)^{\frac{s}{2}}} 
	\frac{V_p(x,y)^{\frac{p}{2}}}{ \left[ 1-X^{\frac{N-p}{2(p-1)}} \right]^{\frac{(p-1)(p-s)}{N-p}}} \,dxdy \\
	&\ge C \, \ep^{-p} \int_{B_{\ep/2}(z_\ep)} \ep^{-\frac{(p-1)(p-s)}{N-p} } dz 
	= O \( \ep^{N-p -\frac{(p-1)(p-s)}{N-p}} \) \quad (\ep \to 0).
\end{align*}
Applying $u_\ep$ \tcr{as a test function for} $S$, \tcr{we see}
\begin{align*}
	S \le C \, \ep^{(N-p) \(1- \frac{p}{p^*(s)} \) + \frac{(p-1) (p-s) p}{(N-p) p^*(s)}} = O\( \ep^{\frac{N-1}{N-s} (p-s)} \) \to 0\,\,\text{as}\,\, \ep \to 0.
\end{align*}
\end{proof}
}

\tcr{Next proposition is a fact from linear algebra.}

\begin{prop}\label{Prop detJ}
Let $v \in \re^N$, $|v|=1$, $t \in \re$, $A= I + t \,v \otimes v$, \tcr{where $I$ is the identity matrix on $\re^N$.} 
Then $A$ has two eigenvalues $1$ and $1+t$. 
The multiplicity of $1$ is $N-1$, and the multiplicity of $1+t$ is $1$. 
Especially, 
\begin{align*}
	\det A = 1+t.
\end{align*} 
If $t \not=-1$, then there exists the inverse matrix 
\begin{align}\label{A^{-1}}
	A^{-1} = I - \frac{t}{t+1} \,v \otimes v.
\end{align}
\end{prop}

\begin{proof}
Let $u =(u_1, \dots, u_N)^T$ satisfy $u \cdot v =0$. 
Then we have
\begin{align*}
	\( (v \otimes v ) u \)_i =  \sum_{j=1}^N (v \otimes v)_{i,j} u_j = \sum_{j=1}^N v_i  v_j u_j 
	= \( \sum_{j=1}^N  v_j u_j \) v_i = 0.
\end{align*}
Therefore, $A u = (I + t v \otimes v) u = u$ which means that $u$ is the \tcr{eigenvector} of the eigenvalue $1$ of $A$. 
\tcr{Note that there are $N-1$ such linearly independent $u$, thus the multiplicity of the eigenvalue $1$ is $N-1$.}
\tcr{Also} since $|v|=1$, for any $i = 1, \dots, N$, we have
\begin{align*}
	\( (v \otimes v ) v \)_i & = \sum_{j=1}^N v_i  v_j v_j = \( \sum_{j=1}^N  v_j v_j \) v_i = v_i
\end{align*}
which implies that $(v \otimes v) v = v$. 
Therefore, $A v = (I + t v \otimes v) v = (1 + t) v$ which means that $v$ is the \tcr{eigenvector} of the eigenvalue $1+t$ of $A$. 
Hence, 
the multiplicity of $1+t$ is $1$ and ${\rm det} (\tcr{I} + t v \otimes v) = 1 + t$. 

Next, we show (\ref{A^{-1}}). 
Since 
\begin{align*}
	\tcr{\((v \otimes v)^2\)_{ik}} &=  \sum_{j=1}^N (v \otimes v)_{i,j} (v \otimes v)_{jk} = \sum_{j=1}^N v_i v_j v_j v_k \\
	&= \( \sum_{j=1}^N  v_j^2 \) v_i v_k =  (v \otimes v)_{ik},
\end{align*}
we have $(v \otimes v)^2 =  (v \otimes v)$. Therefore, we have 
$$
	(I + t v \otimes v)(\tcr{I} - s v \otimes v) = I + (t-s-st) v \otimes v. 
$$
If $t-s-st = 0$, then the right-hand side is $I$. 
Therefore, $A^{-1} = I - \frac{t}{t+1} v \otimes v$. 
\end{proof}


\begin{prop}\label{Prop Cayley is Mobius}
Let $J, T_b, S_\la$ be given by Definition \ref{def M} and ${\bf B}$ be given by (\ref{Cayley type}). 
Then
\begin{align*}
	{\bf B} (z) = R \circ J \circ T_{e_N} \circ S_2 \circ J \circ T_{-e_N} (z), \text{where}\,\,R
	= 
	\begin{pmatrix}
	1 & &  & \\
	& \ddots &  & \\
	& & 1 & \\
	& & & -1
	\end{pmatrix},\,
	e_N = \begin{pmatrix}
	0 \\
	\vdots \\
	0 \\
	1
\end{pmatrix}
\end{align*}
for any $z \in \re^N$. 
\end{prop}

\begin{proof}
Direct calculation shows that
\begin{align*}
J \circ T_{e_N} \circ S_2 \circ J \circ T_{-e_N} (z) 
&=[e_N + 2(z - e_N)^* ]^* = \frac{e_N + 2(z - e_N)^*}{|e_N + 2(z - e_N)^*|^2} \\
&= \frac{e_N + \frac{2(z - e_N)}{|z - e_N|^2}}{\left| e_N + \frac{2(z - e_N)}{|z - e_N |^2} \right|^2} \\
&= \frac{e_N + \frac{2(z - e_N)}{|z - e_N|^2}}{1 + \frac{4}{|z - e_N |^2} + \frac{4 e_N \cdot (z - e_N)}{|z - e_N |^2}}  \\
&= \frac{e_N |z-e_N|^2 + 2(z - e_N)}{|z - e_N|^2 + 4 z_N}  \\
&= \frac{e_N (|z|^2 -2 z_N + 1) + 2(z - e_N)}{|z|^2 + 2 z_N + 1}  \\
&= \frac{(2x, |z|^2 -1)}{|x|^2 + (1 + y)^2}  \\
&= \frac{(2x, |x|^2 + y^2 -1)}{|x|^2 + (1 + y)^2}.
\end{align*}
Therefore, we have
\begin{align*}
R \circ J \circ T_{e_N} \circ S_2 \circ J \circ T_{-e_N} (z)
= \frac{(2x, 1- |x|^2 - y^2 )}{|x|^2 + (1 + y)^2}
= {\bf B} (z).
\end{align*}
\end{proof}


Finally, we describe an application of the transformations in III.-(i) in \S \ref{S Harmonic}. 
For more general case, see \cite{HK}. 
It is well-known that the Sobolev inequality 
\begin{align}\label{Sobolev}
	\tcr{S_{N,p}} \( \int_{\re^N} |u|^{p^*}\,dx \)^{\frac{p}{p^*}} \le \int_{\re^N} |\nabla u|^p\,dx
\end{align}
for any $u \in \dot{W}_0^{1,p}(\re^N)$, $1 < p < N$, with the best constant
$$
	\tcr{S_{N,p}} = \pi^{\frac{p}{2}} N \( \frac{N-p}{p-1} \)^{p-1} 
	\( \frac{\Gamma(\frac{N}{p})\Gamma(1 + N - \frac{N}{p})}{\Gamma(1 + \frac{N}{2}) \Gamma(N)} \)^{\frac{p}{N}}\tcr{,}
$$
follows from a one-dimensional inequality obtained by Bliss \cite{Bliss}:
Let $v: [0, +\infty) \to \re$ be an absolutely continuous function on $(0, +\infty)$ such that $v' \in L^p(0,+\infty)$, $v(0) = 0$.
Put $q > p > 1$.
Then the inequality
\begin{equation}
\label{Bliss}
	C(p,q) \( \int_0^{\infty} \frac{|v(t)|^q}{t^{1 + q(\frac{p-1}{p})}} dt \)^{1/q} \le  \( \int_0^{\infty} |v'(t)|^p dt \)^{1/p}
\end{equation}
holds where
$$
	C(p,q) = \( \frac{\Gamma\(\frac{q}{q-p}\)\Gamma\(\frac{p(q-1)}{q-p}\)}{\Gamma\(\frac{pq}{q-p}\)} \)^{1/p-1/q} \( \frac{q(p-1)}{p} \)^{1/q}.
$$
See Maz'ya \cite{Mazya}, pp. 274, the equation (4.6.4).
Also see \cite{Os} for a new proof of this classical inequality.
In the following, we consider radial functions only.
To obtain the best Sobolev inequality (\ref{Sobolev}) for radial functions $u(r)$, $r = |x|$, we change the variables 
$$
	v(t) = u(r), \,\text{where}\,\,t = \frac{p-1}{N-p} \w_{N-1}^{-\frac{1}{p-1}} r^{-\frac{N-p}{p-1}}
$$
and apply the Bliss inequality (\ref{Bliss}) for $q = p^* > p > 1$.
Note that the condition $v(0) = 0$ is satisfied if $u(r) \to 0$ as $r \to \infty$.

\vspace{1em}
Instead, let us change the variables 
$$
	v(t) = u(r), \,\text{where}\,\, t = \w_{N-1}^{-\frac{1}{N-1}} \log \frac{1}{r}
$$
and put $q > p = N$.
Then the usual computation shows that
\begin{align*}
	&\omega_{N-1}^{1-\frac{q}{N}} \int_0^{\infty} \frac{|v(t)|^q}{t^{1 + q(\frac{N-1}{N})}} dt = \omega_{N-1} \int_0^1  \frac{|u(r)|^q}{r \( \log \frac{1}{r} \)^{1 + q(\frac{N-1}{N})}} \, dr
	= \int_{B_1^N}  \frac{|u(x)|^q}{|x|^N \( \log \frac{1}{|x|} \)^{1 + q(\frac{N-1}{N})}} dx, \\
	&\int_0^{\infty} |v'(t)|^N dt = \int_{B_1^N} |\nabla u|^N dx.
\end{align*}
Also in this case, the condition $v(0) = 0$ is equivalent to $u(1) = 0$.
Inserting these identities into \eqref{Bliss}, we obtain
\begin{equation}
\label{Bliss2}
	C(q) \( \int_{B_1^N}  \frac{|u(x)|^q}{|x|^N \( \log \frac{1}{|x|} \)^{1 + q(\frac{N-1}{N})}} dx \)^{N/q} \le \int_{B_1^N} |\nabla u|^N dx
\end{equation}
for any $u \in \dot{W}^{1,N}_{0, rad}(B_1^N)$ where
\begin{equation}
\label{C_q}
	C(q) = \omega_{N-1}^{1-N/q} C(N, q)^{-N} 
	= \omega_{N-1}^{1-N/q} \( \frac{\Gamma\(\frac{q}{q-N}\)\Gamma\(\frac{N(q-1)}{q-N}\)}{\Gamma\(\frac{Nq}{q-N}\)} \)^{1-N/q} 
	\( \frac{q(N-1)}{N} \)^{N/q}.
\end{equation}
For the inequality (\ref{Bliss2}), see e.g. \cite{S(JDE)}. 
Now, we check that $\lim_{q \to N + 0} C(q) = \( \frac{N-1}{N} \)^N$.
Recall the Stirling formula 
$$
	\Gamma(s) = \sqrt{2\pi} s^{s-1/2} e^{-s} + o(1), \quad \text{as} \quad s \to +\infty 
$$ 
and put $s = \frac{q}{q-N}$.
Then, we see
\begin{align*}
	&\frac{\Gamma\(\frac{q}{q-N}\)\Gamma\(\frac{N(q-1)}{q-N}\)}{\Gamma\(\frac{Nq}{q-N}\)} = \frac{\Gamma(s)\Gamma(\frac{(q-1)N}{q} s)}{\Gamma(Ns)} 
	\sim \frac{\Gamma(s)\Gamma((N-1)s)}{\Gamma(Ns)} \\
	&\sim \sqrt{2\pi} \frac{s^{s-1/2} e^{-s} ((N-1)s))^{(N-1)s-1/2} e^{-(N-1)s}}{(Ns)^{Ns-1/2} e^{-Ns}}
	= \sqrt{2\pi} \frac{(N-1)^{(N-1)s-1/2}}{N^{Ns-1/2}} s^{-1/2}
\end{align*}
\tcr{as $q \to N+0$ (which is equivalent to $s= \frac{q}{q-N} \to \infty$),} 
and for $C(q)$ in \eqref{C_q}, we have
\begin{align*}
	C(q) &\sim \( \frac{\Gamma(s)\Gamma(\frac{N(q-1)}{q} s)}{\Gamma(Ns)} \)^{1/s} \( \frac{q(N-1)}{N} \)^{N/q} \\
	&\sim \( \sqrt{2\pi} \frac{(N-1)^{(N-1)s-1/2}}{N^{Ns-1/2}} s^{-1/2} \)^{1/s} (N-1) \\ 
	&\sim \frac{(N-1)^{N-1}}{N^N} (N-1) s^{-1/2s} \to \( \frac{N-1}{N} \)^N \quad \text{as} \quad \tcr{s \to +\infty.} 
\end{align*}
Thus if we take a limit $q \to N + 0$ in the inequality \eqref{Bliss2}, we have 
the critical Hardy inequality
$$
	\( \frac{N-1}{N} \)^N \int_{B_1^N}  \frac{|u(x)|^N}{|x|^N \( \log \frac{1}{|x|} \)^N} dx \le \int_{B_1^N} |\nabla u|^N dx
$$
on a unit ball for any $u \in \dot{W}^{1,N}_{0, {\rm rad}}(B_1^N)$. 

In conclusion, we obtain the following.
\begin{prop}\label{Prop Bliss}
The Bliss inequality (\ref{Bliss}) \tcr{yields} both the best Sobolev inequality (\ref{Sobolev}) 
and the generalized critical Hardy \tcr{inequality} (\ref{Bliss2}) for radially symmetric functions. 
\end{prop}


\section*{Acknowledgment}
The first author (M.S.) was supported by JSPS KAKENHI Early-Career Scientists, No. JP19K14568. 
The second author (F.T.) was supported by JSPS Grant-in-Aid for Scientific Research (B), No. JP19136384. 
This work was partly supported by Osaka City University Advanced
Mathematical Institute (MEXT Joint Usage/Research Center on Mathematics
and Theoretical Physics) \tcbb{JPMXP0619217849}.


\end{document}